\documentclass{amsart}
\usepackage{amsfonts}
\usepackage{amsmath}
\usepackage{amssymb}
\usepackage{mathrsfs}
\usepackage[all]{xy}

\newtheorem{theorem}{Theorem}[section]
\newtheorem{lemma}[theorem]{Lemma}
\newtheorem{corollary}[theorem]{Corollary}
\newtheorem{proposition}[theorem]{Proposition}
\theoremstyle{definition}

\theoremstyle{remark}
\newtheorem{remark}[theorem]{Remark}

\numberwithin{equation}{section}

\author{Xiaoyan Zhou}
\author{Junsheng Fang}

\address{Xiaoyan Zhou}\address{School of Mathematical Sciences, Dalian University of Technology. Dalian~{116024}. China}
\email{doctoryan@mail.dlut.edu.cn}

\address{Junsheng Fang}\address{School of Mathematical Sciences, Dalian University of Technology. Dalian~{116024}. China}
\email{junshengfang@hotmail.com}

\begin{document}
\title[]{A note on relative amenability of finite von Neumann algebras}

\maketitle

\begin{abstract}
Let $M$ be a finite von Neumann algebra (resp. a type II$_{1}$ factor) and let $N\subset M$ be a II$_{1}$ factor (resp. $N\subset M$ have an atomic part). We prove that the inclusion $N\subset M$ is amenable implies the identity map on $M$ has an approximate factorization through $N\otimes M_m(\mathbb{C})$ via trace preserving normal unital completely positive maps, which is a generalization of a result of Haagerup. We also prove two permanence properties for amenable inclusions. One is weak Haagerup property, the other is weak exactness.
\end{abstract}
\maketitle

\section{Introduction}
To study operator algebras analogue of the rigidity phenomena in representation of groups and ergodic theory, Connes \cite{connes-80, connes-82, connes-jones}
introduced the key concept of correspondences between two von Neumann algebras, which can be thought of as the representation theory for von Neumann algebras.
He also observed that there are many ways to look at these correspondences. For example, we can construct a correspondence $H_{\phi}$ from a normal completely positive map $\phi$ (on a finite von Neumann algebra) using Stinespring dilation and vice versa. Later on, Popa \cite{popa} systematically
developed the theory of correspondences to get new insights in the structure of von Neumann algebras, especially in the study of type II$_1$ factors.

In this paper, we are interested in a relative notion of amenability Popa introduced using the correspondence framework. Recall that for a von Neumann subalgebra $N$ of a finite von Neumann algebra $M$, we say that the inclusion $N\subset M$ is \emph{amenable} (or $M$ is \emph{amenable relative} to $N$, or $N$ is \emph{co-amenable} in $M$) if $H_{id}$ is weakly contained in $H_{E_N}$, where $E_N$ is the trace preserving normal conditional expectation from $M$ onto $N$. Here are some examples of amenable inclusions. If $M$ is a finite von Neumann algebra, then $M$ is amenable if and only if the inclusion $\mathbb{C}1 \subset M$ is amenable. If $N\subset M$ is an inclusion of II$_{1}$ factors, and the Jones' index $[M:N]<\infty$, then the inclusion $N \subset M$ is amenable. If $M$ is a cocycle crossed product of a finite von Neumann algebra $N$ by a cocycle action of a discrete group $G$, then the inclusion $N \subset M$ is amenable if and only if $G$ is an amenable group. If $N$ is a finite von Neumann algebra and $G\curvearrowright N$ is a weakly compact action, then the inclusion L$G\subset N\rtimes G$ is amenable by \cite[Proposition 3.2]{op}.

There are some permanence results for amenable inclusions.  B\'edos \cite{erik}  proved that if $G$ is a discrete amenable group with a free action $\alpha$ on a von Neumann algebra $M$ and $M$ has property $\Gamma$, then $M\rtimes_{\alpha}G$ has property $\Gamma$.  The author also proved that if $G$ is a discrete amenable group with a free action $\alpha$ on a type II$_{1}$ factor $M$ and $M$ is McDuff, then $M\rtimes_{\alpha}G$ is McDuff. Bannon and Fang \cite{fang} proved that if the inclusion of finite von Neumann algebras $N\subset M$ is amenable and $N$ has the Haagerup property, then $M$ also has the Haagerup property.

Just as many other conditions are equivalent to amenability, Popa showed the relative amenability can be characterized by the corresponding ``relative type'' conditions, see \cite[Theorem 3.23]{popa}. Since semidiscreteness is equivalent to amenability for von Neumann algebras, Popa asked whether a good analogue notion exists for relative amenability. This was answered affirmatively by Mingo in \cite{mingo} for finite von Neumann algebras using normal completely positive maps, which is close to the definition of semidiscreteness in spirit. More precisely, he showed that for a
finite von Neumann algebra $M$ and two normal completely positive maps $\phi$, $\varphi: M\rightarrow M$, $H_{\phi}$ is weakly contained in $H_{\varphi}$ if and only if $\phi$ can be approximately factored by $\varphi$. Later on, Anantharaman-Delaroche extended Mingo's result to all von Neumann algebras using correspondences in \cite{anan}.


Applying Mingo's above result and the definition of approximate factorization, it is not difficult to deduce the following proposition.

\begin{proposition}[see Proposition \ref{1st theorem}]\label{starting prop from Mingo}
Let $M$ be a finite von Neumann algebra with a faithful normal trace $\tau$, and let $N\subset M$ be a von Neumann subalgebra. If the inclusion $N\subset M$
is amenable, then there exists a net of normal u.c.p. maps $\varphi_{i}:M\rightarrow M_{n_{i}}(\mathbb{C})
\otimes N$, a net of normal u.c.p. maps $\phi_{i}:M_{n_{i}}(\mathbb{C})\otimes N\rightarrow M$ and
a net of positive elements $h_{i}\in M_{n_{i}}(\mathbb{C})\otimes N$ such that for all $x\in M$, $y\in M_{n_{i}}(\mathbb{C})\otimes N$,
\begin{enumerate}
\item $\phi_{i}\circ\varphi_{i}(x)\rightarrow x$ in the $\|\cdot\|_{2}$-norm topology,
\item $\tau\circ\phi_{i}(y)=(tr_{n_{i}}\otimes\tau)(h_{i}y)$.
\end{enumerate}
\end{proposition}

We may try to apply Proposition 1.1 to study permanence properties for amenable inclusions, i.e., we try to prove if some approximation property holds for a von Neumann subalgebra $N$, then it also holds for the finite von Neumann algebra $M$ assuming the inclusion $N\subset M$ is amenable. However, it turns out that in several situations, we need to assume $h_i$ to be the identity; in other words, we expect the normal u.c.p. maps $\phi_i$, $\varphi_i$ can be chosen to be trace preserving. In fact, this issue also appears in Haagerup's proof that semidiscreteness $\Rightarrow$ hyperfiniteness for a II$_1$ factor, see \cite{haagerup}. Under certain assumptions on the two algebras, we show $\phi_i$, $\varphi_i$ could be chosen to be trace preserving.

The following are our main theorems.

\begin{theorem}[see Theorem \ref{diffuse}]
Let $M$ be a finite von Neumann algebra with a faithful normal tracial state $\tau$, and let $N\subset M$ be a type II$_{1}$ factor. Let the inclusion $N\subset M$
be amenable. Let $\{x_{1},\ldots,x_{n}\}$ be a finite set in $M$ and let $\varepsilon>0$.
Then there exists an $m\in \mathbb{N}$, and two normal c.p. maps
$S:M\rightarrow M_{m}(\mathbb{C})\otimes N,$ $T:M_{m}(\mathbb{C})\otimes N\rightarrow M,$
such that
\begin{enumerate}
\item $S$ and $T$ are unital,
\item $(tr_{m}\otimes\tau)\circ S=\tau,\quad \tau\circ T=tr_{m}\otimes\tau$,
\item $\|T\circ S(x_{k})-x_{k}\|_{2}<\varepsilon,~k=1,\ldots,n$.
\end{enumerate}
\end{theorem}

\begin{theorem}[see Theorem \ref{atomic}]
Let $M$ be a type II$_{1}$ factor with a faithful normal tracial state $\tau$, and let $N\subset M$ be a von Neumann subalgebra having an atomic part. Let the inclusion $N\subset M$
be amenable. Let $\{x_{1},\ldots,x_{n}\}$ be a finite set in $M$ and let $\varepsilon>0$.
Then there exists an $m\in \mathbb{N}$, and two normal c.p. maps
$S:M\rightarrow M_{m}(\mathbb{C})\otimes N,$ $T:M_{m}(\mathbb{C})\otimes N\rightarrow M,$
such that
\begin{enumerate}
\item $S$ and $T$ are unital,
\item $(tr_{m}\otimes\tau)\circ S=\tau,\quad \tau\circ T=tr_{m}\otimes\tau$,
\item $\|T\circ S(x_{k})-x_{k}\|_{2}<\varepsilon,~k=1,\ldots,n$.
\end{enumerate}
\end{theorem}

Since $M$ is amenable if and only if the inclusion $\mathbb{C}1\subset M$ is amenable (c.f. \cite[3.23]{popa} or \cite[Proposition 5]{popa-monod}), Theorem 1.3 generalizes a result of Haagerup \cite[Proposition 3.5]{haagerup}, which corresponds to the case $N=\mathbb{C}1$.

Using these two theorems, we could prove some permanence results for amenable inclusions.
\begin{corollary}[See Corollary \ref{corollary on permanence results}]\label{H under amenable inclusion}
Let $M$ be a finite von Neumann algebra and let $N\subset M$ be a type II$_{1}$ factor. If the inclusion $N\subset M$ is amenable and $N$ has the Haagerup property, then $M$ also has the Haagerup property.
\end{corollary}

\begin{corollary}[See Corollary \ref{weakly}]
Let $M$ be a finite von Neumann algebra and let $N\subset M$ be a type $II_{1}$ factor. If the inclusion $N\subset M$ is amenable and $N$ is weakly exact, then $M$ is also weakly exact.
\end{corollary}

\begin{corollary}[See Corollary \ref{weak}]
Let $M$ be a finite von Neumann algebra and let $N\subset M$ be a type II$_{1}$ factor. If the inclusion $N\subset M$ is amenable and $N$ has the weak Haagerup property, then $M$ also has the weak Haagerup property.
\end{corollary}

Note that Bannon and Fang \cite{fang} proved a permanence result for the Haagerup property for amenable inclusions for finite von Neumann algebras in the framework of correspondences. In this paper, we prove Corollary \ref{H under amenable inclusion} from the point of view of normal u.c.p. maps.


This paper is organised as follows. In Section 2, we present some preliminaries.
In Section 3, we prove that the amenability of the inclusion $N\subset M$ of finite von Neumann algebras implies that the identity map on $M$ has an approximate factorization through $N\otimes M_m(\mathbb{C})$ via normal unital completely positive maps.
In Section 4, we use some matrix techniques and the results in Section 3 to show that the above normal unital completely positive maps can be chosen to be trace preserving in two cases: when $M$ is a finite von Neumann algebra and $N\subset M$ is a II$_1$ factor, and when $M$ is a II$_1$ factor and $N\subset M$ has an atomic part. In the last section, we present three permanence properties for some amenable inclusions.

\section{Preliminaries}
In this section, we recall briefly some basic concepts that will be used later. For more details and results on correspondences, relative amenability, and completely positive maps, we refer the reader to \cite{anan4, anan3, anan, anan5, anan6,mingo2, mingo, popa}.

\subsection*{Correspondences}
Let $M$ and $N$ be von Neumann algebras. Recall that \emph{a correspondence from $M$ to $N$} is a $*$-representation of $N\otimes M^{op}$ on a Hilbert space $H$, which is normal when restricted to both $N=N\otimes1$ and $M^{op}=1\otimes M^{op}$.

\subsection*{Correspondences associated to completely positive maps}
Let $M$ be a finite von Neumann algebra with a faithful normal trace $\tau$. Given a normal completely positive map $\phi:M\rightarrow M$, we can use the Stinespring dilation to construct a correspondence which is denoted by $H_{\phi}$. Define on the linear space $H_{0}=M\otimes M$ a sesquilinear form
$\langle x_{1}\otimes y_{1},x_{2}\otimes y_{2} \rangle_{\phi}=\tau(\phi(x^{*}_{2}x_{1})y_{1}y_{2}^{*})$, $\forall x_{1}, y_{1},x_{2},y_{2}\in M$.
It is easy to check that the complete positivity of $\phi$ is equivalent to the positivity of $\langle \cdot,\cdot\rangle_{\phi}$. Let $H_{\phi}$
be the completion of $H_{0}/\sim$, where $\sim$ is the equivalence modulo the null space of $\langle \cdot,\cdot\rangle_{\phi}$. Then $H_{\phi}$ is a correspondence of $M$ and the bimodule structure is given by $x(x_{1}\otimes y_{1})y=xx_{1}\otimes y_{1}y$. We call $H_{\phi}$ the correspondence of $M$
associated to $\phi$, see \cite{popa}.

\subsection*{Relative amenability}
If we regard correspondences as $*$-representations, we can define a topology on these correspondences which is just the usual topology on the set of equivalent classes of representations of $N\otimes M^{op}$. Under this topology, we say that a correspondence \emph{$H_{1}$ is weakly contained in $H_{2}$} if $H_{1}$ is in the closure of $H_{2}$.

Let $M$ be a finite von Neumann algebra with a trace $\tau$, and let $N$ be a von Neumann subalgebra of $M$. Then the inclusion \emph{$N\subset M$ is amenable} if $H_{id}$ is weakly contained in $H_{E_{N}}$, where $id$ is the identity map from $M$ to $M$ and $E_{N}$ is the faithful normal conditional expectation from $M$ onto $N$ preserving trace $\tau$. Popa has given several equivalent conditions for relative amenability in \cite[3.23]{popa} and \cite[Proposition 5]{popa-monod}.

Here are some examples of amenable inclusions. If $M$ is a finite von Neumann algebra, then $M$ is amenable if and only if the inclusion $\mathbb{C}1 \subset M$ is amenable. If $N\subset M$ is an inclusion of II$_{1}$ factors, and the Jones' index $[M:N]<\infty$, then the inclusion $N \subset M$ is amenable. If $M$ is a cocycle crossed product of a finite von Neumann algebra $N$ by a cocycle action of a discrete group $G$, then the inclusion $N \subset M$ is amenable if and only if $G$ is an amenable group. If $N$ is a finite von Neumann algebra and $G\curvearrowright N$ is a weakly compact action, then the inclusion L$G\subset N\rtimes G$ is amenable by \cite[Proposition 3.2]{op}.

\subsection*{Approximate factorization}
Let $\psi:M\rightarrow M$ be completely positive and $a_{1},\ldots$, $a_{n}$, $b_{1},\ldots,b_{n}\in M$. Define
$$\Theta:M\rightarrow M,$$
$$x\mapsto\sum^{n}_{i,j=1}b_{i}^{*}\psi(a_{i}^{*}x a_{j})b_{j}.$$

Let $$ A=\left(\begin{array}{ccc}
a_{1} \ldots a_{n}\\
\end{array}
\right),B=\left(\begin{array}{c}
b_{1}\\
\vdots\\
b_{n}\\
\end{array}
\right).$$
Then $\Theta$ is completely positive by the commutativity of the diagram
$$ \xymatrix{
    M \ar[rr]^{\Theta}\ar[dr]_{\varphi} & & M  \\
     & M_{n}(\mathbb{C})\otimes M\ar[ur]_{\phi} &
    },
$$
where $\varphi(x)=(id_{n}\otimes\psi)(A^{*}xA)$, $\phi(y)=B^{*}yB$, $x\in M$ and
$y \in M_{n}(\mathbb{C})\otimes M$.

We shall say that \emph{a c.p. map $\Theta$ can be factored by $\psi$} if it is of the above form, see \cite{mingo}.
We shall denote by $F_{\psi}$ the set of finite sums of such maps.

Let $\phi,\varphi: M\rightarrow M$ be normal c.p. maps.  that \emph{$\varphi$ may be approximately
factored by $\phi$} if there is a bounded net $(\phi_{r}(x))\subset F_{\phi}$ such that for each $x\in M$,
${\phi_{r}}(x)$ converges to $\varphi(x)$ $\sigma$-weakly for all $x\in M$, see \cite{mingo}.

\subsection*{Haagerup property}
Let $M$ be a finite von Neumann algebra with a faithful normal trace $\tau$. For each $x\in M$, denote by $\|x\|^{2}_2=\tau(x^*x)$.

A finite von Neumann algebra $M$ with a faithful normal trace $\tau$ has the \emph{Haagerup property} if there exists a net $(\phi_{i})_{i\in I}$ of normal
completely positive maps from $M$ to $M$ which satisfy the following conditions,
\begin{enumerate}
\item  $\tau\circ\phi_{i}\leq\tau$,
\item each $\phi_{i}$ induces a compact bounded operator on $L^{2}(M)$,
\item for every $x\in M$, $\lim_{i}\|\phi_{i}(x)-x\|_{2}=0$.
\end{enumerate}
Note that a normal c.p. map $\phi_{i}:M\rightarrow M$ with  $\tau\circ\phi_{i}\leq\tau$ can induce a bounded linear operator on $L^{2}(M)$. To see this,  $\|\phi_{i}(x)\|_{2}^{2}=\tau(\phi_{i}(x)^{*}\phi_{i}(x))\leq\tau(\phi_{i}(x^{*}x))\leq\tau(x^{*}x)=\|x\|_{2}^{2}$. Thus $\phi_{i}$ can be extended to a bounded linear operator on $L^{2}(M)$.

\subsection*{Weak Haagerup property \cite{Knudby}}
Let $M$ be a von Neumann algebra with a faithful normal trace $\tau$.  $M$ has the \emph{weak Haagerup property} if there exist a constant $C>0$ and a net $(\phi_{i})_{i\in I}$ of normal completely bounded maps on $M$ such that
\begin{enumerate}
\item $\|\phi_{i}\|_{c.b.}\leq C$ for every $i$,
\item $\langle \phi_{i}(x),y\rangle_{\tau}=\langle x,\phi_{i}(y)\rangle_{\tau}$ for every $x,y\in M$;
\item each $\phi_{i}$ induces a compact bounded operator on $L^{2}(M)$,
\item for every $x\in M$, $\lim_{i}\|\phi_{i}(x)-x\|_{2}=0$.
\end{enumerate}

\subsection*{Weakly exact von Neumann algebras \cite{ozawabook}}
Let $B$ be an arbitrary unital C*-algebra and $J\lhd B$ be a non-unital closed two-sided ideal. The canonical quotient map will be denoted by $Q:B\rightarrow B/J$.

A von Neumann algebra $M$ is said to be \emph{weakly exact} if for any ideal $J\lhd B$ and any $*$-representation $\pi:M\otimes B \rightarrow B(H)$ with $M\otimes J\subset\textrm{ker}\pi$ and $\pi|_{M\otimes \mathbb{C}1}$ being normal, the induced representation $\tilde{\pi}:M\odot(B/J)\rightarrow B(H)$ is continuous with respect to the minimal tensor norm.

\begin{theorem}[\cite{ozawa}]
Let $M$ be a von Neumann algebra. The following conditions are equivalent.
\begin{enumerate}
\item $M$ is weakly exact.
\item For any finite dimensional operator system $E$ in $M$, there exist two nets of u.c.p. maps $\phi_{i}:E\rightarrow M_{n}(\mathbb{C})$ and $\psi_{i}:\phi_{i}(E)\rightarrow M$ such that the net $(\psi_{i}\circ\phi_{i})_{i\in I}$ converges to $id_{E}$ in the point-$\sigma$-weak operator topology.
\end{enumerate}
\end{theorem}
\begin{remark}\label{different topologies coincide}
Assume that $M$ is a finite von Neumann algebra with a trace $\tau$. Note that the above $(\psi_{i}\circ\phi_{i})_{i\in I}$ are u.c.p. maps. Then the choice of topology in which the net $(\psi_{i}\circ\phi_{i})_{i\in I}$ converges to the identity map on $E$ could be one of many topologies without affecting the results. The topologies are the point-weak operator topology, the point-$\sigma$-weak operator topology, the point-strong operator topology and the pointwise $\|\cdot\|_{2}$-norm topology.
\end{remark}

\section{Approximate factorization of the identity map via normal unital completely positive maps}

As the main result of this section, we prove Proposition \ref{1st theorem}. It is based on a result of Mingo \cite{mingo} on the relation between approximate factorization and weak containment of correspondences.
\begin{theorem}[\cite{mingo}]\label{mingo's theorem}
Let $M$ be a finite von Neumann algebra with a trace $\tau$ and let $\phi,\varphi: M\rightarrow M$ be normal c.p. maps. Then $\varphi$ can be approximately factored by $\phi$
if and only if $H_{\varphi}$ is weakly contained in $H_{\phi}$.
\end{theorem}

For a finite von Neumann algebra $M$ with a faithful normal trace $\tau$, denote by $L^{1}(M)$ the completion of $M$ with respect to the norm $\|x\|_{1}=\tau(|x|)$, $x\in M$.  Note that for the above normal c.p. map $\phi:M_{n}(\mathbb{C})\otimes N\rightarrow M$, we have $\tau\circ\phi(x)=(tr_{n}\otimes\tau)(hx)$, where $tr_{n}\otimes\tau$ is the normal trace on $M_{n}(\mathbb{C})\otimes N$, and $h$ is a positive element in $L^{1}(M_{n}(\mathbb{C})\otimes N)$.

Note that the convergent topology in approximate factorization is the $\sigma$-weak operator topology. The aim of this section is to show that the normal completely positive maps $\phi$ and $\varphi$ in Proposition \ref{1st theorem} can be chosen to be unital, the convergent topology can be the pointwise $\|\cdot\|_{2}$-norm topology, and, the positive element $h$ can be chosen to be invertible in $M_{n}(\mathbb{C})\otimes N$.

We first need the following lemma.
\begin{lemma}\label{lemma 2.4}
Let $M$ be a finite von Neumann algebra with a trace $\tau$ and let $N\subset M$ be a von Neumann subalgebra. Then the inclusion $N\subset M$
is amenable if and only if there exists a net of normal c.p. maps $\varphi_{i}:M\rightarrow M_{n_{i}}(\mathbb{C})\otimes N$ and
a net of normal c.p. maps $\phi_{i}:M_{n_{i}}(\mathbb{C})\otimes N\rightarrow M$ such that
\begin{enumerate}
\item $\varphi_{i}(x)=\bigoplus^{l_{i}}_{j=1}(id_{i_{j}}\otimes E)(A_{ij}^{*}xA_{ij})$ for $x\in M$, $l_{i},i_{j} \in \mathbb{N}$, $A_{ij}\in M_{1\times i_{j}}(M)$,
$\sum_{j=1}^{l_{i}}i_{j}=n_{i}$ and $E$ is the trace preserving normal conditional expectation from $M$ onto $N$,
\item $\phi_{i}(y)=B_{i}^{*}y B_{i}$ for $y\in M_{n_{i}}(\mathbb{C})\otimes N$, $B_{i}\in M_{n_{i}\times 1}(M)$,
\item $\phi_{i}\circ\varphi_{i}(1)\leq 1$,
\item $\phi_{i}\circ\varphi_{i}(x)\rightarrow x$ in the $\|\cdot\|_{2}$-norm topology for all $x\in M$.
\end{enumerate}
\end{lemma}
\begin{proof}
By Theorem \ref{mingo's theorem}, we know that the inclusion $N\subset M$
is amenable if and only if the identity map $id$ can be approximately factored by the normal conditional expectation $E$.

For each element $\Theta$ in $F_{E}$, $\Theta(x)=\sum^{n}_{k=1}\theta_{k}(x)$, where
$$\theta_{k}(x)=\sum^{m^{}_{k}}_{i,j=1}b_{ki}^{*}E(a_{ki}^{*}xa_{kj})b_{kj}, a_{ki}, b_{kj}\in M.$$
For simplicity, we may assume $n=2$. Let
$$ A_{1}=\left(\begin{array}{ccc}
a_{11}&\dots&a_{1m_{1}}\\
\end{array}\right),A_{2}=\left(\begin{array}{ccc}
a_{21}&\dots&a_{2m_{2}}\\
\end{array}
\right),B=\left(\begin{array}{c}
b_{11}\\
\vdots\\
b_{1m_{1}}\\
b_{21}\\
\vdots\\
b_{2m_{2}}
\end{array}
\right).$$
Let $$\varphi(x)=\oplus^{2}_{i=1}(id_{m_{i}}\otimes E)(A_{i}^{*}xA_{i}), x\in M,$$
$$\psi(y)=B^{*}yB,\quad y\in M_{m_{1}+m_{2}}(\mathbb{C})\otimes N.$$
Note that $\varphi$ and $\psi$ are normal c.p. maps from $M$ to $M_{m_{1}+m_{2}}(\mathbb{C})\otimes N$
and $M_{m_{1}+m_{2}}(\mathbb{C})\otimes N$ to $M$ respectively with $\Theta(x)=\psi\circ\varphi(x)$.

It is clear that $F_{E}$ is a convex set and $b^{*}\Theta(\cdot) b\in F_{E}$ for $b\in M$, $\Theta\in F_{E}$.
Then by \cite[Lemma 2.2]{anan} and Theorem \ref{mingo's theorem}, we can choose a net $(\Theta_{i})\subset F_{E}$
such that $\Theta_{i}(1)\leq 1$ and $\Theta_{i}(x)\rightarrow x$ $\sigma$-weakly for all $x\in M$.
Let $F_{E}'=\{\Theta\in F_E: \Theta(1)\leq 1\}$. Obviously, $F_{E}'$ is convex. Note that for a convex set of $CP(M)$, where $CP(M)$ denotes the set of c.p. maps on $M$, the closure in the point-$\sigma$-weak operator topology and the closure in the point-$\sigma$-strong operator topology are the same. And since $F_{E}'$ is bounded, we deduce that $||\Theta_{i}(x)-x||_2\to 0$ for all $x\in M$ for a net $(\Theta_i)\subset F_E'$. Actually, the choice of topology in which the net $(\Theta_{i})$ converges to the identity
map on $M$ could be one of many topologies without affecting the results. The topologies are the point-weak operator topology, the point-$\sigma$-weak operator topology, the point-strong operator topology and the point-wise $\|\cdot\|_{2}$-norm topology.
\end{proof}

\begin{proposition}\label{1st theorem}
Let $M$ be a finite von Neumann algebra with a trace $\tau$ and let $N\subset M$ be a von Neumann subalgebra. If the inclusion $N\subset M$
is amenable, then there exists a net of normal u.c.p. maps $\varphi_{i}:M\rightarrow M_{n_{i}}(\mathbb{C})
\otimes N$, a net of normal u.c.p. maps $\phi_{i}:M_{n_{i}}(\mathbb{C})\otimes N\rightarrow M$ and
a net of positive invertible elements $h_{i}\in M_{n_{i}}(\mathbb{C})\otimes N$ such that for all $x\in M$, $y\in M_{n_{i}}(\mathbb{C})\otimes N$,
\begin{enumerate}
\item $\phi_{i}\circ\varphi_{i}(x)\rightarrow x$ in the $\|\cdot\|_{2}$-norm topology,
\item $\tau\circ\phi_{i}(y)=(tr_{n_{i}}\otimes\tau)(h_{i}y)$.
\end{enumerate}
\end{proposition}
\begin{proof}
By Lemma \ref{lemma 2.4}, there exists a net of normal c.p. maps $\tilde{\psi}_{i}:M\rightarrow M_{n_{i}}(\mathbb{C})\otimes N$
and a net of normal c.p. maps $\tilde{\phi}_{i}:M_{n_{i}}(\mathbb{C})\otimes N\rightarrow M$ such that
$\tilde{\phi}_{i}\circ\tilde{\psi}_{i}(x)\rightarrow x$ in the $\|\cdot\|_{2}$-norm topology for all $x\in M$ and
$\tilde{\phi}_{i}\circ\tilde{\psi}_{i}(1)\leq 1$.

We can choose $(\eta_{i}), (\epsilon_{i})\subset \mathbb{R}_{+}$, such that
$\eta_{i}\rightarrow 1$, $\epsilon_{i}\tilde{\phi}_{i}(1)\rightarrow 0$ in the operator norm topology,
and $0<\epsilon_{i}\tilde{\phi}_{i}(1)+\eta_{i}<1$. Then we have
$\tilde{\phi}_{i}\circ(\eta_{i}\tilde{\psi}_{i}(x)+\epsilon_{i})\rightarrow x$ in the $\|\cdot\|_{2}$-norm topology for all $x\in M$ and
$\tilde{\phi}_{i}\circ(\eta_{i}\tilde{\psi}_{i}(1)+\epsilon_{i})< 1$. Define $\tilde{\varphi}_{i}(x):=\eta_{i}\tilde{\psi}_{i}(x)+\epsilon_{i}$
and $\varphi_{i}(x):=\tilde{\varphi}_{i}(1)^{-\frac{1}{2}}\tilde{\varphi}_{i}(x)\tilde{\varphi}_{i}(1)^{-\frac{1}{2}}$. Then $\varphi_{i}$ is a normal u.c.p.
map from $M$ to $M_{n_{i}}(\mathbb{C})\otimes N$.

Let $b_{i}=1-\tilde{\phi}_{i}\circ\tilde{\varphi}_{i}(1)$. Since $\tilde{\phi}_{i}\circ\tilde{\varphi}_{i}(1)<1$,
we have $b_{i}> 0$ and $b_{i}\rightarrow 0$ in the $\|\cdot\|_{2}$-norm topology.

Define linear maps
$\phi_{i}:M_{n_{i}}(\mathbb{C})\otimes N\rightarrow M$ by $$\phi_{i}(y)=(tr_{n_{i}}\otimes\tau)(y)b_{i}+ \tilde{\phi}_{i}(\tilde{\varphi}_{i}(1)^{\frac{1}{2}}y\tilde{\varphi}_{i}(1)^{\frac{1}{2}}).$$
Then the $\phi_{i}'$s are normal u.c.p. maps. Since $b_{i}\rightarrow 0$, it follows that
$\phi_{i}\circ\varphi_{i}(x)\rightarrow x$ in the $\|\cdot\|_{2}$-norm topology.

By Lemma \ref{lemma 2.4}, $\tilde{\phi}_{i}(y)=B_{i}^{*}y B_{i}$ for $y\in M_{n_{i}}(\mathbb{C})\otimes N$, $B_{i}\in M_{n_{i}\times 1}(M)$.

For simplicity, write $n=n_i$ and $\tilde{\phi}_{i}$ from $M_{n}(\mathbb{C})\otimes N$ to $M$ in the following form
$$\tilde{\phi}_{i}(y)=\left(\begin{array}{c}
b_{1}\\
\vdots\\
b_{n}
\end{array}
\right)^{*}
\left(\begin{array}{ccc}
y_{11}&\dots&y_{1n}\\
\vdots&&\vdots\\
y_{n1}&\dots&y_{nn}
\end{array}
\right)\left(\begin{array}{c}
b_{1}\\
\vdots\\
b_{n}
\end{array}
\right)=\sum^{n}_{i,j=1}b_{i}^{*}y_{ij}b_{j},$$
where $b_{i}$ is in $M$ and $y=(y_{ij})_{n\times n}$ is in $M_{n}(\mathbb{C})\otimes N$.

Let $h_{ij}=nb_{j}b_{i}^{*}\in M$ and put $\tilde{h}=(h_{ij})_{n\times n}\in M_{n}(\mathbb{C})\otimes M$. Then we have $\tilde{h}\geq0$ and
$$(tr_{n}\otimes\tau)(\tilde{h}y)=\sum^{n}_{i,j=1}\tau(\dfrac{h_{ij}}{n}y_{ji})=\sum^{n}_{i,j=1}\tau(b_{j}b_{i}^{*}y_{ij})=\tau\circ\tilde{\phi}_{i}(y).$$
Since conditional expectation preserves the trace and $y$ is in $M_{n}(\mathbb{C})\otimes N$, we have
$$(tr_{n}\otimes\tau)(\tilde{h}y)=(tr_{n}\otimes\tau)(E_{M_{n}(\mathbb{C})\otimes N}(\tilde{h}y))=(tr_{n}\otimes\tau)(E_{M_{n}(\mathbb{C})\otimes N}(\tilde{h})y).$$

Note that
\begin{align*}
\tau\circ\phi_{i}(y)=&\tau(b_{i})(tr_{n}\otimes\tau)(y)+ \tau\circ\tilde{\phi}_{i}(\tilde{\varphi}_{i}(1)^{\frac{1}{2}}y\tilde{\varphi}_{i}(1)^{\frac{1}{2}})\\
=&(tr_{n}\otimes\tau)(\tau(b_{i})y+\tilde{\varphi}_{i}(1)^{\frac{1}{2}}E_{M_{n}(\mathbb{C})\otimes N}(\tilde{h})\tilde{\varphi}_{i}(1)^{\frac{1}{2}}y).
\end{align*}

Let $h=\tau(b_{i})+\tilde{\varphi}_{i}(1)^{\frac{1}{2}}E_{M_{n}(\mathbb{C})\otimes N}(\tilde{h})\tilde{\varphi}_{i}(1)^{\frac{1}{2}}$.
Since $\tilde{\varphi}_{i}(1)\in M_{n}(\mathbb{C})\otimes N$, $\tilde{h}\geq0$ and $b_{i}>0$, we have that $h\in M_{n}(\mathbb{C})\otimes N$ is positive and invertible. Hence, we finish the proof.
\end{proof}

\section{Main results}

In this section, we extend Haagerup's result \cite[Proposition 3.5]{haagerup} to amenable inclusions in two cases, either the subalgebra $N$ has an atomic part  and the ambient algebra $M$ is a II$_1$ factor or $N$ is a II$_1$ factor.

The first case follows quite easily from \cite[Proposition 3.5]{haagerup}, while the second case is quite involved.

Recall that a von Neumann algebra $N$ has an atomic part means that there exists a nonzero projection $p\in N$ such that $pNp=\mathbb{C}p$.
\begin{theorem}\label{atomic}
Let $M$ be a type II$_{1}$ factor with a faithful normal tracial state $\tau$, and let $N\subset M$ be a von Neumann subalgebra having an atomic part. Let the inclusion $N\subset M$
be amenable. Let $\{x_{1},\ldots,x_{n}\}$ be a finite set in $M$ and let $\varepsilon>0$.
Then there exists an $m\in \mathbb{N}$, and two normal c.p. maps
$S:M\rightarrow M_{m}(\mathbb{C})\otimes N,$ $T:M_{m}(\mathbb{C})\otimes N\rightarrow M,$
such that
\begin{enumerate}
\item $S$ and $T$ are unital,
\item $(tr_{m}\otimes\tau)\circ S=\tau,\quad \tau\circ T=tr_{m}\otimes\tau$,
\item $\|T\circ S(x_{k})-x_{k}\|_{2}<\varepsilon,~k=1,\ldots,n$.
\end{enumerate}
\end{theorem}
\begin{proof}
Assume $p$ is a projection in $N$ such that $pNp=\mathbb{C}p$.  By \cite[Theorem 3.23]{popa}, we have that $\mathbb{C}p\subset pMp$ is amenable, which shows that $pMp$ is a hyperfinite type II$_1$ factor.
We can find a projection $e$ in $M$ such that $e\leq p$ and $\tau(e)=\frac{1}{k}$ for some positive integer $k$.
It follows that $M$ is a hyperfinite type II$_1$ factor, since $M=M_{k}(\mathbb{C})\otimes eMe$ and $eMe$ is a
hyperfinite type II$_1$ factor.

Let $\{x_{1},\ldots,x_{n}\}$ be a finite set in $M$ and let $\varepsilon>0$. By \cite[Proposition 3.5]{haagerup},
there exists an $m\in \mathbb{N}$, and two normal u.c.p. maps
$S_{1}:M\rightarrow M_{m}(\mathbb{C}),\quad T_{1}:M_{m}(\mathbb{C})\rightarrow M,$
such that $tr_{m}\circ S_{1}=\tau,\quad \tau\circ T_{1}=tr_{m}$
and $\|T_{1}\circ S_{1}(x_{k})-x_{k}\|_{2}<\varepsilon,~k=1,\ldots, n.$

Define two normal unital c.p. maps $S_{2}$ from $M_{m}(\mathbb{C})$ to $M_{m}(\mathbb{C})\otimes N$ and $T_{2}$ from
$M_{m}(\mathbb{C})\otimes N$ to $M_{m}(\mathbb{C})$ respectively by
$$S_{2}(x)=x\otimes 1,\quad T_{2}(y\otimes z)=\tau(z)y,\quad x,y\in M_{m}(\mathbb{C}), z\in M.$$
Put $S=S_{2}\circ S_{1}$, $T=T_{1}\circ T_{2}$. Then
$S:M\rightarrow M_{m}(\mathbb{C})\otimes N,\quad T:M_{m}(\mathbb{C})\otimes N\rightarrow M$
are two normal unital c.p. maps.

Note that for $x\in M$, $y\in M_{m}(\mathbb{C})$ and $z\in N$,
\begin{align*}
(tr_{m}\otimes\tau)(S(x))&=(tr_{m}\otimes\tau)(S_{1}(x)\otimes 1)\\
&=tr_{m}\circ S_{1}(x)\\
&=\tau(x)
\end{align*}
and
\begin{align*}
\tau\circ T(y\otimes z)&=\tau\circ T_{1}(y\tau(z))\\
&=\tau(z)\tau(T_{1}(y))\\
&=(tr_{m}\otimes\tau)(y\otimes z).
\end{align*}
Moreover,
$\|T\circ S(x)-x\|_{2}=\|T_{1}\circ S_{1}(x)-x\|_{2}.$
Hence we finish the proof.
\end{proof}

\begin{theorem}\label{diffuse}
Let $M$ be a finite von Neumann algebra with a faithful normal tracial state $\tau$, and let $N\subset M$ be a type II$_{1}$ factor. Let the inclusion $N\subset M$
be amenable. Let $\{x_{1},\ldots,x_{n}\}$ be a finite set in $M$ and let $\varepsilon>0$.
Then there exists an $m\in \mathbb{N}$, and two normal c.p. maps
$S:M\rightarrow M_{m}(\mathbb{C})\otimes N,$ $T:M_{m}(\mathbb{C})\otimes N\rightarrow M,$
such that
\begin{enumerate}
\item $S$ and $T$ are unital,
\item $(tr_{m}\otimes\tau)\circ S=\tau,\quad \tau\circ T=tr_{m}\otimes\tau$,
\item $\|T\circ S(x_{k})-x_{k}\|_{2}<\varepsilon,~k=1,\ldots,n$.
\end{enumerate}
\end{theorem}

For the sake of proving Theorem \ref{diffuse}, we introduce the following definitions.

For any normal state $\phi$ on a von Neumann algebra $M$, we put
$$\|x\|_{\phi}^{\sharp}=\phi(\frac{x^{*}x+xx^{*}}{2})^{\frac{1}{2}}, \quad \mbox{for~~} x \in M.$$
\emph{A ``good" simple operator} in a type II$_{1}$ factor means an operator with the form $\sum_{i=1}^{n}\lambda_{i}e_{i}$,
where $\lambda_{i}\in \mathbb{C}$ and $e^{}_{1},\ldots, e^{}_{n}$ are equivalent mutually orthogonal projections with $\sum_{i=1}^{n}e_{i}=1$. \emph{A rational positive ``good'' simple operator} is a positive ``good'' simple operator with rational numbers as coefficients.  a ``good'' simple operator $h$ in $M_{m}(\mathbb{C})\otimes N$ is of \emph{``scalar form''} if $h=\sum_{1\leq i\leq m}f_{ii}\otimes \lambda_{i, i}1_{N}$, where $\{f_{ij}\}_{1\leq i,j\leq m}$ are the matrix units in $M_{m}(\mathbb{C})$, $\lambda_{i, i}\in \mathbb{C}$ and $1_{N}$ is the identity operator in $N$.

Our strategy to prove Theorem \ref{diffuse} is to mimic Haagerup's proof of \cite[Proposition 3.5]{haagerup}. To use Haagerup's techniques, we first need Lemma \ref{1 lemma} and Lemma \ref{2}.

 Using Proposition \ref{1st theorem} in our paper, we deduce that for any $\varepsilon>0$, there exist
 two normal u.c.p. maps $S:M\rightarrow M_{n}(\mathbb{C})
 \otimes N$, $T:M_{n}(\mathbb{C})\otimes N\rightarrow M$ such that for all $x\in M$, $\|T\circ S(x)-x\|_{2}< \varepsilon$ and  $\tau\circ T(x)=(tr_{n}\otimes\tau)(hx)$, where $h$ is a positive invertible element in $M_{n}(\mathbb{C})\otimes N$. Then, using a result of Kadison in \cite{Kadison}, we can assume $h$ is of diagonal form in $M_{n}(\mathbb{C})\otimes N$. In Haagerup's situation, $N=\mathbb{C}$, so $h$ is always of scalar form, but in general, this $h$ may not be of scalar form.
 Note that in Haagerup's assumptions, he dealt with $h\in M_{m}(\mathbb{C})$, which is of scalar form.
 If $N$ is a diffuse finite factor, then we can assume that $h$ is a ``good'' simple operator and we can also make a perturbation of $h$ to assume its coefficients to be rational, this is our Lemma \ref{1 lemma}. In Lemma \ref{2}, we amplify $M_{n}(\mathbb{C})\otimes N$ to $M_{k}(\mathbb{C})\otimes M_{n}(\mathbb{C})\otimes N$, and in this larger algebra, $h$ can be written in scalar form. 

\begin{lemma}\label{1 lemma}
Let $M$ be a finite von Neumann algebra with a faithful normal trace $\tau_{M}$, and let $N$ be
 a type II$_{1}$ factor with trace $\tau_{N}$. Let
$T:N\rightarrow M$ be a normal u.c.p. map such that $$\tau_{M}\circ T(y)=\tau_{N}(yh),~~\forall~y\in N,$$
and let $h$ be an invertible positive operator in $N$.
For any $y_{1},\ldots,y_{n}\in N$ and any $\varepsilon >0$, there exists a normal u.c.p. map
$T'$ from $N$ to $M$ such that $$\|T(y_{i})-T'(y_{i})\|_{2}<\varepsilon \mbox{~~and~~}\tau_{M}\circ T'(y)=\tau_{N}(h'y)$$ for $1\leq i\leq n$
 and all $y\in N$, where $h'$ is an invertible rational positive ``good'' simple operator in $N$.
\end{lemma}
\begin{proof}
Since $h$ is an invertible positive operator in the type II$_{1}$ factor $N$, we can identify $h$ with a positive
function $h(t)$,
$0\leq t\leq 1$ and assume that $h(t)\geq \delta>0$ for all $t$. Since $N$ is a type II$_{1}$ factor, there exists a sequence of ``good'' simple operators
$h_{k}=h_{k}(t)$ with the property that
\begin{enumerate}
\item $\delta\leq h_{k}(t)\leq h(t)$ for all $t$, $0\leq t \leq 1$;
\item $lim_{k\rightarrow\infty}h_{k}(t)=h(t)$ for almost all $t$, $0\leq t\leq 1$.
\end{enumerate}

Assume $\|h-h_{k}\|_{1}<\varepsilon$ for some $\varepsilon>0$. Let $b_{k}=b_{k}(t)=\frac{h_{k}(t)}{h(t)}$. Then $0<b_{k}(t)\leq 1$ for all $0\leq t\leq 1$.  Note that
$$\|1-b_{k}\|_{1}=\tau(1-b_{k})=\int_{0}^{1}\frac{h(t)-h_{k}(t)}{h(t)}dt\leq\frac{1}{\delta}\|h-h_{k}\|_{1}<\frac{\varepsilon}{\delta},$$
and $$\|1-b_{k}\|_{2}^{2}=\tau((1-b_{k})^{2})=\int_{0}^{1}\frac{(h(t)-h_{k}(t))^{2}}{(h(t))^{2}}dt\leq\frac{2\|h\|}{\delta^{2}}
\|h-h_{k}\|_{1}<\frac{2\|h\|}{\delta^{2}}
\varepsilon.$$

Define $T_{k}:N\rightarrow M$ by
$$T_{k}(y)=T(b_{k}^{\frac{1}{2}}yb_{k}^{\frac{1}{2}})+\tau_{N}(y)T(1-b_{k}),\quad \mbox{for~~} y\in N.$$
Then $T_{k}$ is a normal u.c.p. map. Note that $b_{k}$ commutes with $h$, so for $y\in N$, we deduce
\begin{align*}
\tau_{M}\circ T_{k}(y)=&\tau_{M}\circ T(b_{k}^{\frac{1}{2}}yb_{k}^{\frac{1}{2}})+\tau_{N}(y)\tau_{M}(T(1-b_{k}))\\
=&\tau_{N}(hb_{k}^{\frac{1}{2}}yb_{k}^{\frac{1}{2}})+\tau_{N}(y)\tau_{N}(h(1-b_{k}))\\
=&\tau_{N}(h'_{k}y),
\end{align*}
where $h'_{k}=hb_{k}+\tau_{N}(h(1-b_{k}))1=h_{k}+\tau_{N}(h(1-b_{k}))1$
is an invertible positive ``good'' simple operator.

By the Schwartz inequality for c.p. maps, we have for $y\in N$,
$$\|T(y)\|_{2}=\tau_{M}(T(y^{*})T(y))^{\frac{1}{2}}\leq\tau_{M}(T(y^{*}y)))^{\frac{1}{2}}=\tau_{N}(hy^{*}y)^{\frac{1}{2}}\leq\|
h\|^{\frac{1}{2}}\|y\|_{2}.$$
By \cite[Proposition 1.2.1]{cones}, we have $\|1-b_{k}^{\frac{1}{2}}\|_{2}\leq \|1-b_{k}\|^{\frac{1}{2}}_{1}$.
Moreover, for $1\leq i\leq n$,
\begin{align*}
\|T_{k}(y_{i})-T(y_{i})\|_{2}&\leq\|T(y_{i}-b_{k}^{\frac{1}{2}}y_ib_{k}^{\frac{1}{2}})\|_{2}+|\tau_{N}(y_{i})|\|T(1-b_{k})\|_{2}\\
&\leq\|T(y_{i}(1-b_{k}^{\frac{1}{2}}))\|_{2}+\|T((1-b_{k}^{\frac{1}{2}})y_{i}b_{k}^{\frac{1}{2}})\|_{2}+|\tau_{N}(y_{i})|\|T(1-b_{k})\|_{2}\\
&\leq\|h\|^{\frac{1}{2}}(\|y_{i}(1-b_{k}^{\frac{1}{2}})\|_{2}+\|(1-b_{k}^{\frac{1}{2}})y_{i}b_{k}^{\frac{1}{2}}\|_{2}+|\tau_{N}(y_{i})|\|1-b_{k}\|_{2})\\
&\leq2\|h\|^{\frac{1}{2}}\|y_{i}\|\|1-b_{k}\|^{\frac{1}{2}}_{1}+|\tau_{N}(y_{i})|\|h\|^{\frac{1}{2}}\|1-b_{k}\|_{2}\\
&\rightarrow 0.
\end{align*}

Next we want to make a perturbation of the invertible positive ``good'' simple operator $h_{k}'$ to get rational coefficients.

Note that $h_{k}'\in N$ is an invertible positive ``good'' simple operator and $\tau_{M}\circ T_{k}(1)=\tau_{N}(h'_{k})=1$. Let $\lambda_{1},\ldots,\lambda_{m}$ be the diagonal elements of $h_{k}'$. Then we have $\lambda_{i}>0\mbox{~~and~~} \sum_{i=1}^{m}\lambda_{i}=m.$

Choose rational numbers $q_{1},\ldots,q_{m}$ such that $(1-\varepsilon)\lambda_{i}<q_{i}<\lambda_{i}$. Put $u_{i}=\frac{q_{i}}{\lambda_{i}}$ for $i=1,\ldots,m$.
Moreover, let $s$ be the diagonal matrix with the diagonal elements $u_{1},\ldots,u_{m}$. Then $1-\varepsilon <s<1$.
Define a map $T'$ from $N$ to $M$
by $$T'(x)=T_{k}(s^{\frac{1}{2}}xs^{\frac{1}{2}})+\tau_{N}(x)T_{k}(1-s).$$ Then $T'$ is a normal u.c.p. map and
\begin{align*}
\|T_{k}(x)-T'(x)\|&\leq\|s^{\frac{1}{2}}xs^{\frac{1}{2}}-x\|+\|1-s\|\|x\|\\
&=\frac{1}{2}\|(1+s^{\frac{1}{2}})x(1-s^{\frac{1}{2}})+(1-s^{\frac{1}{2}})x(1+s^{\frac{1}{2}})\|+\|1-s\|\|x\|\\
&\leq(\|1+s^{\frac{1}{2}}\|\|1-s^{\frac{1}{2}}\|+\|1-s\|)\|x\|\\
&<3\varepsilon\|x\|.
\end{align*}
We have $\|T'-T_{k}\|\rightarrow 0 \mbox{~and~}(\tau\circ T')(x)=\tau_{N}(h'x),$ where $h'=s^{\frac{1}{2}}h_{k}'s^{\frac{1}{2}}+\tau_{N}((h_{2}(1-s))$.
Let $l_{1},\ldots,l_{m}$ be the diagonal elements of $h'$.
Note that $\tau_{N}(h_{k}'s)=\sum_{i=1}^{m}\frac{q_{i}}{m}$. Then we have $l_{i}=q_{i}+(1-\sum_{i=1}^{m}\frac{q_{i}}{m})>0$ and is rational.

Then for $1\leq i\leq n$, we get
\begin{align*}
\|T(y_{i})-T'(y_{i})\|_{2}&\leq \|T(y_{i})-T_{k}(y_{i})\|_{2}+\|T_{k}(y_{i})-T'(y_{i})\|_{2}\\
& \rightarrow 0.
\end{align*}Hence we finish the proof.
\end{proof}

\begin{lemma}\label{2}
Let $M$ be a finite von Neumann algebra with a faithful normal tracial state $\tau$, and let $N\subset M$ be a type II$_{1}$ factor . Let the inclusion $N\subset M$
be amenable. Let $\{x_{1},\ldots,x_{m}\}$ be a finite set in $M$ and let $\varepsilon>0$.
Then there exists an $n\in \mathbb{N}$, and two normal u.c.p. maps
$S:M\rightarrow M_{n}(\mathbb{C})\otimes N,$ $T:M_{n}(\mathbb{C})\otimes N\rightarrow M,$
such that
\begin{enumerate}
\item $\tau\circ T(y)=(tr_{n}\otimes\tau)(hy)$, where $y,h\in M_{n}(\mathbb{C})\otimes N$ and $h$ is an invertible rational positive ``good'' simple operator, furthermore, it is of ``scalar form'',
\item $\|T\circ S(x_{i})-x_{i}\|_{2}<\varepsilon,~i=1,\ldots,m$.
\end{enumerate}
\end{lemma}
\begin{proof}
By Proposition \ref{1st theorem}, for any $\varepsilon>0$ we can find two normal u.c.p. maps $S_{1}:M\rightarrow M_{n}(\mathbb{C})\otimes N,$ $T_{1}:M_{n}(\mathbb{C})\otimes N\rightarrow M,$
such that $\tau\circ T_{1}(y)=(tr_{n}\otimes\tau)(h_{1}y)$, $\|T_{1}\circ S_{1}(x_{i})-x_{i}\|_{2}<\varepsilon$, where $h_{1},y\in M_{n}(\mathbb{C})\otimes N$, $h_{1}$ is an invertible positive operator and $i=1,\ldots,m$.
By Lemma \ref{1 lemma}, we have a normal u.c.p. map $T_{2}:M_{n}(\mathbb{C})\otimes N\rightarrow M,$ with $\tau\circ T_{2}(y)=(tr_{n}\otimes\tau)(h_{2}y)$,
where $h_{2},y\in M_{n}(\mathbb{C})\otimes N$ and $h_{2}$ is an invertible rational positive ``good'' simple operator.

By the definition of ``good'' simple operators, assume $h_{2}=\sum_{i=1}^{k}\lambda_{i}e_{i},$ where $\{\lambda_{i}\}$ are positive rational numbers and $\{e_{i}\}$ are equivalent mutually orthogonal projections with $\sum_{i=1}^{k}e_{i}=1$. Note that there exists a transform $U$ of $M_{k}(\mathbb{C})\otimes M_{n}(\mathbb{C})\otimes N$ which turns $I_{k}\otimes h_{2}$ into a ``scalar form''. Write $U(z)=vzv^{*}$, where $v,z\in M_{k}(\mathbb{C})\otimes M_{n}(\mathbb{C})\otimes N$, $v$ is some unitary element, and $h:=U(I_{k}\otimes h_{2})$. Then $h$ is an invertible rational positive ``good'' simple operator; furthermore, it is of ``scalar form''.

Define $T=T_{2}\circ (tr_{k}\otimes id_{M_{n}(\mathbb{C})\otimes N})\circ U^{-1}$ and $S=U\circ (id_{k}\otimes id_{M_{n}(\mathbb{C})\otimes N})\circ S_{1}$, where $id_{M_{n}(\mathbb{C})\otimes N}$ is the identity map on $M_{n}(\mathbb{C})\otimes N$, $id_{k}$ is the identity map on $M_{k}(\mathbb{C})$. It is clear that $\|T\circ S(x_{i})-x_{i}\|_{2}<\varepsilon$, $i=1,\ldots,m$.

Let $v=\sum_{1\leq i,j\leq k}e_{ij}\otimes x_{ij}$, where $\{e_{ij}\}_{1\leq i,j\leq k}\subset  M_{k}(\mathbb{C})$ are the matrix units and $x_{ij}\in M_{n}(\mathbb{C})\otimes N$. Then for  $a\in M_{k}(\mathbb{C}), x\in M_{n}(\mathbb{C})\otimes N$, we have
\begin{align*}
\tau\circ T(a\otimes x)&=(tr_{n}\otimes \tau)(h_{2}(tr_{k}\otimes id_{M_{n}(\mathbb{C})\otimes N}) U^{-1}(a\otimes x))\\
&=(tr_{n}\otimes \tau)(h_{2}(tr_{k}\otimes id_{M_{n}(\mathbb{C})\otimes N})(\sum_{i,j,s,t}e_{ji}ae_{st}\otimes x^{*}_{ij}xx_{st}))\\
&=(tr_{n}\otimes \tau)(h_{2}\sum_{i,j,s}tr_{k}(e_{si}a)x^{*}_{ij}xx_{sj})\\
&=\sum_{i,j,s}tr_{k}(e_{si}a)(tr_{n}\otimes \tau)(h_{2}x^{*}_{ij}xx_{sj}),
\end{align*}

\begin{align*}
(tr_{k}\otimes tr_{n}\otimes \tau)(h(a\otimes x))&=(tr_{k}\otimes tr_{n}\otimes \tau)(v(I_{k}\otimes h_{2})v^{*}(a\otimes x))\\
&=(tr_{k}\otimes tr_{n}\otimes \tau)(\sum_{i,j,s}e_{is}\otimes x_{ij}h_{2}x^{*}_{sj}(a\otimes x))\\
&=\sum_{i,j,s}tr_{k}(e_{si}a)(tr_{n}\otimes \tau)(h_{2}x^{*}_{ij}xx_{sj}).
\end{align*}
Thus we have $\tau(T(a\otimes x))=(tr_{k}\otimes tr_{n}\otimes \tau)(h( a\otimes x))$, where $a\in M_{k}(\mathbb{C})$, $x\in M_{n}(\mathbb{C})\otimes N$. Let $m=nk$.
Hence we finish the proof.
\end{proof}

With the help of the above two lemmas, we will mimic \cite[Lemma 3.1, Lemma 3.2]{haagerup} to prove the following two lemmas which also generalise \cite[Lemma 3.1, Lemma 3.2]{haagerup}. We should mention that the proofs are not trivial. We have to overcome some new difficulties since under our assumptions we deal with $M_{m}(\mathbb{C})\otimes N$ where $N$ is a von Neumann algebra, while Haagerup dealt with $M_{m}(\mathbb{C})$.

 The difficulty of Lemma \ref{3} is Claim A, i.e., $S$ maps $M$ into $M_{m}(\mathbb{C})\otimes N$, and it is normal.

\begin{lemma}\label{3}
Let $M$ be a finite von Neumann algebra with a faithful normal trace $\tau$ and $N\subset M$ be a von Neumann subalgebra.
Let $m\in \mathbb{N}$ and $T$ be a normal u.c.p. map from $M_{m}(\mathbb{C})\otimes N$ to $M$
such that $(\tau\circ T)(x)=(tr_{m}\otimes \tau)(hx),$ where $h$ is an invertible positive element in $M_{m}(\mathbb{C})\otimes N$.
Put $\phi(x)=\tau\circ T(x),\mbox{~~for~~} x \in M_{m}(\mathbb{C})\otimes N.$
Then
\begin{enumerate}
\item There is a unique normal u.c.p. map $S$ from $M$ to $M_{m}(\mathbb{C})\otimes N$
such that $$(tr_{m}\otimes\tau)(h^{\frac{1}{2}}S(y)h^{\frac{1}{2}}x^{*})=\tau(yT(x)^{*})$$
for all $y\in$ $M$ and all $x\in M_{m}(\mathbb{C})\otimes N$. Moreover,
$\phi\circ S(y)=\tau(y)~ \mbox{for~} y\in M.$
\item For all $x\in M_{m}(\mathbb{C})\otimes N$, $\|T(x)\|_{2}^{2}\leq (tr_{m}\otimes\tau)
(h^{\frac{1}{2}}x h^{\frac{1}{2}}x^{*}).$
\end{enumerate}
\end{lemma}
\begin{proof}
\begin{enumerate}
\item If $S_{1}$, $S_{2}$ satisfy the condition in (1), then for $y\in$ $M$,
$$(tr_{m}\otimes\tau)(h^{\frac{1}{2}}S_{1}(y)h^{\frac{1}{2}}x^{*})=(tr_{m}\otimes\tau)
(h^{\frac{1}{2}}S_{2}(y)h^{\frac{1}{2}}x^{*})$$
for all $x \in M_{m}(\mathbb{C})\otimes N$. This implies that
$h^{\frac{1}{2}}S_{1}(y)h^{\frac{1}{2}}=h^{\frac{1}{2}}S_{2}(y)h^{\frac{1}{2}}$ and consequently
$S_{1}(y)=S_{2}(y)$ since $h$ is invertible.

Let $s$ be the inner product on $M_{m}(\mathbb{C})\otimes N$ defined by for $x_{1}, x_{2} \in M_{m}(\mathbb{C})\otimes N,$
$s(x_{1},x_{2})=(tr_{m}\otimes\tau)(h^{\frac{1}{2}}x_{1} h^{\frac{1}{2}}x^{*}_{2}).$

Note that $s$ is positive definite because
$$s(x_{1},x_{2})=(tr_{m}\otimes\tau)((h^{\frac{1}{4}}x_{1} h^{\frac{1}{4}})(h^{\frac{1}{4}}x_{2} h^{\frac{1}{4}})^{*}).$$

For $x\in M_{m}(\mathbb{C})\otimes N$, we have
$$\|T(x)\|_{2}^{2}=\tau(T^{*}(x)T(x))\leq\tau(T(x^{*}x))=(tr_{m}\otimes\tau)(hx^{*}x).$$
Moreover,
\begin{align*}
(tr_{m}\otimes\tau)(hx^{*}x)&=(tr_{m}\otimes\tau)(h^{\frac{1}{2}}x^{*}h^{\frac{1}{4}}h^{-\frac{1}{2}}h^{\frac{1}{4}}xh^{\frac{1}{2}})\\
&\leq\|h^{-\frac{1}{2}}\|(tr_{m}\otimes\tau)(h^{\frac{1}{2}}
x^{*}h^{\frac{1}{4}}h^{\frac{1}{4}}xh^{\frac{1}{2}})\\
&=\|h^{-\frac{1}{2}}\|(tr_{m}\otimes\tau)(h^{\frac{1}{2}}
h^{\frac{1}{4}}x^{*}h^{\frac{1}{4}}h^{\frac{1}{4}}xh^{\frac{1}{4}})\\
&\leq\|h^{-\frac{1}{2}}\|\|h^{\frac{1}{2}}\|(tr_{m}\otimes\tau)(
h^{\frac{1}{4}}x^{*}h^{\frac{1}{4}}h^{\frac{1}{4}}xh^{\frac{1}{4}})\\
&=\|h^{-\frac{1}{2}}\|\|h^{\frac{1}{2}}\|\|x\|^{2}_{s}.\\
\end{align*}
Denote by $(M_{m}(\mathbb{C})\otimes N,s)$ the completion of $M_{m}(\mathbb{C})\otimes N$ with respect to the norm induced by the inner product $s$.
Thus there exists a bounded linear map $T_{0}$ from the Hilbert space
$(M_{m}(\mathbb{C})\otimes N,s)$ to the Hilbert space $L^{2}(M,\tau)$ with the restriction to be $T$ on $M_{m}(\mathbb{C})\otimes N$.

Let
$T_{0}^{*}:L^{2}(M,\tau)\rightarrow (M_{m}(\mathbb{C})\otimes N,s)$ be the adjoint operator
and let $S$ be the restriction of $T_{0}^{*}$ to $M$.

Claim A: $S$ is a normal map which maps $M$ into $M_{m}(\mathbb{C})\otimes N$.

\emph{Proof of Claim A}.
For $x \in (M_{m}(\mathbb{C})\otimes N)_{+},~ y \in M_{+}$,
\begin{align*}
\tau(yT(x))&=\tau(T(x)^{\frac{1}{2}}yT(x)^{\frac{1}{2}})\\
&\leq\|y\|\tau\circ T(x)\\
&=\|y\|(tr_{m}\otimes\tau)(hx)\\
&=\|y\|(tr_{m}\otimes\tau)(x^{\frac{1}{2}}hx^{\frac{1}{2}})\\
&\leq\|y\|\|h\|(tr_{m}\otimes\tau)(x).
\end{align*}

Note that for any fixed $y$ in $M_{+}$, $\tau(yT(x))$ and $(tr_{m}\otimes\tau)(x)$ are normal positive linear functionals on $M_{m}(\mathbb{C})\otimes N$.
By \cite[Theorem 7.3.6]{Kadison book}, there exists a positive element $z$ in $M_{m}(\mathbb{C})\otimes N$ such that
$\tau(yT(x))=(tr_{m}\otimes\tau)(xz)$.
Besides, since $h$ is invertible, we have
$$(tr_{m}\otimes\tau)(xz)=(tr_{m}\otimes\tau)(h^{\frac{1}{2}}h^{-\frac{1}{2}}zh^{-\frac{1}{2}}h^{\frac{1}{2}}x)=s(h^{-\frac{1}{2}}zh^{-\frac{1}{2}},x^{*}).$$
For $x \in M_{m}(\mathbb{C})\otimes N, y \in M$,
$$s(S(y),x)=s(T_{0}^{*}(y),x)=(y,T_{0}(x))_{\tau}=\tau(yT(x^{*})).$$
Then we can obtain that for $x \in (M_{m}(\mathbb{C})\otimes N)_{+}, y \in M_{+}$,
$$s(S(y),x)=\tau(yT(x))=s(h^{-\frac{1}{2}}zh^{-\frac{1}{2}},x),$$
which implies $S(y)=h^{-\frac{1}{2}}zh^{-\frac{1}{2}}$ and hence $S$ is normal.
Since $h$ and $z$ are both in $M_{m}(\mathbb{C})\otimes N$, $S$ maps all the elements of $M$ into $M_{m}(\mathbb{C})\otimes N$. This ends the proof of Claim A.

It is clear that $$(tr_{m}\otimes\tau)(h^{\frac{1}{2}}S(1) h^{\frac{1}{2}}x^{*})=s(S(1),x)=\tau(T(x)^{*})=(tr_{m}\otimes\tau)(hx^{*}),$$
hence $S(1)=1$ since $h$ is invertible. For $y\in N$, we have $$\phi\circ S(y)=\tau\circ T\circ S(y)=
(tr_{m}\otimes\tau)(hS(y))=s(S(y),1)=\tau(y).$$

To prove that $S$ is completely positive, we will need the fact that an operator $x$ in a finite von Neumann algebra $B$ is positive
if and only if $\tau_{B}(xy)\geq 0$ for any $y\in B_{+}$.
Here, $\tau_{B}$ is a faithful normal tracial state on $B$.

Let $n\in$ $\mathbb{N}$, $(e_{ij})_{i,j=1,\ldots,n}$ be the matrix units in $M_{n}(\mathbb{C})$. Let
$I_{n}$ be the identity in $M_{n}(\mathbb{C})$. Put $S^{(n)}=I_{n}\otimes S$, $T^{(n)}=I_{n}\otimes T$.
We shall prove that $S^{(n)}$ is a positive map for all $n\in$ $\mathbb{N}$. Let
$a=\sum_{i,j=1}^{n}e_{ij}\otimes a_{ij}\in M_{n}(\mathbb{C})\otimes M,$ and
$b=\sum_{i,j=1}^{n}e_{ij}\otimes b_{ij}\in M_{n}(\mathbb{C})\otimes(M_{m}(\mathbb{C})\otimes N).$

Then
\begin{align*}
&(tr_{n}\otimes(tr_{m}\otimes\tau))((I_{n}\otimes h^{\frac{1}{2}})S^{(n)}(a)(I_{n}\otimes h^{\frac{1}{2}})b^{*})\\
=&(tr_{n}\otimes(tr_{m}\otimes\tau))((\sum_{i,j=1}^{n}e_{ij}\otimes h^{\frac{1}{2}}S(a_{ij})h^{\frac{1}{2}})(\sum_{s,t=1}^{n}e_{ts}\otimes b_{st}^{*} ))\\
=&\frac{1}{n}(tr_{m}\otimes\tau)(\sum_{i,j=1}^{n}h^{\frac{1}{2}}S(a_{ij})h^{\frac{1}{2}}b_{ij}^{*})\\
=&\frac{1}{n}\sum_{i,j=1}^{n}s(S(a_{ij}),b_{ij})\\
=&\frac{1}{n}\sum_{i,j=1}^{n}\tau(a_{ij}T(b_{ij}^{*}))\\
=&(tr_{n}\otimes\tau)(aT^{(n)}(b)^{*}).\\
\end{align*}
For all $a\in( M_{n}(\mathbb{C})\otimes M)_{+}$ and $b\in (M_{n}(\mathbb{C})\otimes(M_{m}(\mathbb{C})\otimes N))_{+}$, we have $(I_{n}\otimes h^{\frac{1}{2}} )S^{(n)}(a)(I_{n}\otimes h^{\frac{1}{2}})\in  (M_{n}(\mathbb{C})
\otimes(M_{m}(\mathbb{C})\otimes N))_{+} $ since $T^{(n)}$ is positive.
Hence $S^{(n)}$ is a positive map.
\item
The composed map $T\circ S$ is a normal u.c.p. map from $M$ to $M$ and $\tau\circ(T\circ S)=\phi\circ S=\tau$.
Then $\|T\circ S(x)\|_{2}\leq\|x\|_{2}$ using the Schwartz inequality for c.p. maps. Hence $\|T_{0}\circ T_{0}^{*}\|\leq 1$,
where $T_{0}$ is the map $T$ considered as a linear map from the Hilbert space $(M_{m}(\mathbb{C})\otimes N,s)$ to $L^{2}(N,\tau)$.
Thus $\|T_{0}\|^{2}=\|T_{0}\circ T_{0}^{*}\|\leq1$, i.e. $\|T(x)\|_{2}^{2}\leq s(x,x)=(tr_{m}\otimes\tau)
(h^{\frac{1}{2}}x h^{\frac{1}{2}}x^{*})$, $x \in M_{m}(\mathbb{C})\otimes N$.
\end{enumerate}
 \end{proof}

To prove Lemma \ref{4}, we first use the same method as Haagerup did to prove Claim A. The difficulty in our proof is Claim B. In Haagerup's proof, he first constructed a u.c.p. map $T:M_{m}(\mathbb{C})\rightarrow M_{q}(\mathbb{C})$ which is claim A in our proof, then he used \cite[Lemma 3.1]{haagerup} to get a u.c.p. map $S: M_{q}(\mathbb{C})\rightarrow M_{m}(\mathbb{C})$. Since this $S$ is defined abstractly, to estimate $S\circ T(e_{ij})$, he used the fact that $x\in M_m(\mathbb{C})$ is determined once we know $tr_{m}(xe_{ij})$ for all the matrix units $\{e_{ij}\}_{1\leq i,j\leq m}$ in $M_{m}(\mathbb{C})$.
However in our situation, this method does not work. Instead, to prove claim B, we directly construct a normal u.c.p. map $S:M_{q}(\mathbb{C})\otimes N\rightarrow M_{m}(\mathbb{C})\otimes N$ such that for $x_{ij}\in N$, $S\circ T(e_{ij}\otimes x_{ij})$ can be estimated.

\begin{lemma}\label{4}
Let $M$ be a finite von Neumann algebra with a faithful normal trace $\tau$ and let $N\subset M$ be a von Neumann subalgebra.
Let $\phi$ be a normal state on $M_{m}(\mathbb{C})\otimes N$ of the form $$\phi(x)=(tr_{m}\otimes\tau)(hx),$$
where $h$ is an invertible rational positive ``good'' simple operator, and it is of ``scalar form'' in $M_{m}(\mathbb{C})\otimes N$. Then there exists a $q\in \mathbb{N}$, and two normal u.c.p. maps $T:M_{m}(\mathbb{C})\otimes N\rightarrow M_{q}(\mathbb{C})\otimes N$,
$S:M_{q}(\mathbb{C})\otimes N\rightarrow M_{m}(\mathbb{C})\otimes N$
such that
\begin{enumerate}
\item $\phi\circ S=tr_{q}\otimes\tau,~(tr_{q}\otimes\tau)\circ T=\phi$,
\item $\|S\circ T(x)-x\|_{\phi}^{\sharp}\leq\|h^{\frac{1}{2}}x-xh^{\frac{1}{2}}\|_{2},~x\in M_{m}(\mathbb{C})\otimes N$.
\end{enumerate}
\end{lemma}
\begin{proof}
Claim A: there exists a normal u.c.p. map $T:M_{m}(\mathbb{C})\otimes N\rightarrow M_{q}(\mathbb{C})\otimes N$ such that $(tr_{q}\otimes\tau)\circ T=\phi$.

\emph{Proof of Claim A}.
Assume $h$ is of the diagonal form with diagonal elements $\lambda_{1},\ldots,\lambda_{m}$, where $\lambda'_{i}s$ are strictly
positive rational numbers. Then we can choose positive integers $p_{1},\ldots,p_{m}$ and $q$ such that
$\frac{\lambda_{i}}{m}=\frac{p_{i}}{q},~~i=1,\ldots,m.$
Since $(tr_{m}\otimes\tau)(h)=1$, we have $\sum_{i=1}^{m}p_{i}=q$.

A $q\times q$-matrix $y$ can be represented by a block matrix
$y=(y_{ij})_{i,j=1,\ldots,m},$ where each $y_{ij}$ is a $p_{i}\times p_{j}$-matrix. Let $F_{ij}$ denote the $p_{i}\times p_{j}$-matrix
given by
\begin{equation*}
(F_{ij})_{k,l}=\begin{cases}
1,&  \text{if $k=l$,} \\
0,&   \text{if $k\neq l$.}\\
\end{cases}
\end{equation*}
and let $f_{ij}$ denote the $q\times q$-matrix with block matrix
\begin{equation*}
(f_{ij})_{i'j'}=\begin{cases}
F_{ij}, &\text{if $(i', j')=(i, j)$,}\\
0,& \text{otherwise.}\\
\end{cases}
\end{equation*}
Note that the number $1$ occurs  min$\{p_{i},p_{j}\}$ times in $F_{ij}$ and $f_{ij}$. Let $(e_{ij})_{i,j=1,\ldots,m}$ be the matrix units in
$M_{m}(\mathbb{C})$ and define a linear map $T$ from $M_{m}(\mathbb{C})\otimes N$ to $M_{q}(\mathbb{C})\otimes N$ by
$T(\sum_{i,j=1}^{m}e_{ij}\otimes x_{ij})=\sum_{i,j=1}^{m}f_{ij}\otimes x_{ij},~~x_{ij}\in N.$
Then $T$ is unital. Moreover, for $i\neq j$, we have
\begin{align*}
(tr_{q}\otimes\tau )(T(e_{ij}\otimes x_{ij}))=( tr_{q}\otimes\tau)( f_{ij}\otimes x_{ij})=( tr_{m}\otimes\tau)(h(e_{ij}\otimes x_{ij}))=0,\\
(tr_{q}\otimes\tau)(T(e_{ii}\otimes x_{ii}))=tr_{q}(f_{ii})\tau(x_{ii})= \frac{\lambda_{i}}{m}\tau(x_{ii})=
( tr_{m}\otimes\tau)(h( e_{ii}\otimes x_{ii})).
\end{align*}
Hence, $( tr_{q}\otimes\tau)\circ T(x)=(tr_{m}\otimes\tau )(hx)=\phi(x),~~x\in M_{m}(\mathbb{C})\otimes N.$
To see that $T$ is completely positive, put $p=max\{p_{1},\ldots,p_{m}\}$ and let $\tilde{f}_{ij}$ be the element in $M_{mp}(\mathbb{C})$
given by the $m\times m$-block matrix
\begin{equation*}
(\tilde{f}_{ij})_{i'j'}=\begin{cases}
I_p, & \text{if $(i', j')=(i, j)$,} \\
0, &\text{otherwise.}\\
\end{cases}
\end{equation*}
Here $I_{p}$ is the $p\times p$-unit matrix. The map $\tilde{T}$ from $M_{m}(\mathbb{C})\otimes N$ to $M_{mp}(\mathbb{C})\otimes N$ by
$\tilde{T}(\sum_{i,j=1}^{m}e_{ij}\otimes x_{ij})=\sum_{i,j=1}^{m}\tilde{f}_{ij}\otimes x_{ij},~~x_{ij}\in N,$
is a $*$-representation and therefore completely positive.
It is not difficult to see that there exists a projection $e$ in $M_{mp}(\mathbb{C})\otimes N$ such that
$e(M_{mp}(\mathbb{C})\otimes N)e=M_{q}(\mathbb{C})\otimes N$ and
$T(x)=e\tilde{T}(x)e,~~x\in M_{m}(\mathbb{C})\otimes N.$
Hence $T$ is normal and completely positive.
This ends the proof of Claim A.

Claim B: there is a normal u.c.p. map $S:M_{q}(\mathbb{C})\otimes N\rightarrow M_{m}(\mathbb{C})\otimes N$ such that $\phi\circ S=tr_{q}\otimes\tau$ and $S\circ T(e_{ij}\otimes x_{ij})=\frac{min\{p_{i},p_{j}\}}{\sqrt{p_{i}p_{j}}}e_{ij}\otimes x_{ij}.$

\emph{Proof of Claim B}.
For any $s,t\in\mathbb{N}$, define a linear map $D$ from  $M_{s\times t}(\mathbb{C})\otimes N$ to $N$ by
$$D(\sum_{1\leq i\leq s, 1\leq j\leq t}l_{ij}\otimes h_{ij})=\sum_{i=1}^{min\{s,t\}}h_{ii},$$
where $(l_{ij})_{1\leq i\leq s, 1\leq j\leq t}$ is the matrix units in $M_{s\times t}(\mathbb{C})$ and $h_{ij}$ is in $N$
for any $1\leq i\leq s, 1\leq j\leq t$.
Let $(k_{st})_{s,t=1,\ldots,q}$ be the matrix units in $M_{q}(\mathbb{C})$. For $x=\sum_{i,j=1}^{m}e_{ij}\otimes x_{ij}\in
 M_{m}(\mathbb{C})\otimes N
,~~y=\sum_{i,j=1}^{q} k_{ij}\otimes y_{ij}\in M_{q}(\mathbb{C})\otimes N,$ define a linear map $S'$ from
 $M_{q}(\mathbb{C})\otimes N$ to $M_{m}(\mathbb{C})\otimes N$ by
 $$S'(y)=\sum_{i,j=1}^{m} e_{ij}\otimes\frac{1}{\sqrt{p_{i}p_{j}}} D(f_{ii}yf_{jj}).$$
 For $1\leq i,j \leq m$, put $a_{ij}=\frac{1}{\sqrt{p_{i}p_{j}}} D(f_{ii}yf_{jj})$
 and $p_{0}=0$, then
 $$a_{ij}=\frac{1}{\sqrt{p_{i}p_{j}}}\sum_{k=1}^{min\{p_{i},p_{j}\}}y_{p_{1}+p_{2}+\ldots+p_{i-1}+k,~p_{1}+p_{2}+\ldots+p_{j-1}+k}.$$

Note that
\begin{align*}
(tr_{m}\otimes\tau)(h^{\frac{1}{2}}S'(y) h^{\frac{1}{2}}x^{*})&=(tr_{m}\otimes\tau)((\sum_{i,j=1}^{m} e_{ij}\otimes\sqrt{\lambda_{i}\lambda_{j}}a_{ij})(\sum_{k,l=1}^{m}e_{lk}\otimes x_{kl}^{*}))\\
&=\sum_{i,j=1}^{m}\frac{
 \tau(\sqrt{\lambda_{i}\lambda_{j}}a_{ij}x^{*}_{ij})}{m}\\
 &=\sum_{i,j=1}^{m}\frac{
 \tau(\sqrt{p_{i}p_{j}}a_{ij}x^{*}_{ij})}{q}\\
 &=\sum_{i,j=1}^{m}\sum_{k=1}^{min\{p_{i},p_{j}\}}\frac{\tau(y_{p_{1}+p_{2}+\ldots+p_{i-1}+k,~p_{1}+p_{2}+
 \ldots+p_{j-1}+k}x^{*}_{ij})}{q}.
\end{align*}
Note that $f_{ij}=\sum_{k=1}^{min\{p_{i},p_{j}\}}k_{p_{1}+p_{2}+\ldots+p_{i-1}+k,~p_{1}+p_{2}+\ldots+p_{j-1}+k},$
then we have
\begin{align*}
(tr_{q}\otimes\tau)(yT(x)^{*})&=(tr_{q}\otimes\tau)((\sum_{s,t=1}^{q} k_{st}\otimes y_{st})(\sum_{i,j=1}^{m}f_{ji}\otimes x_{ij}^{*}))\\
&=\sum_{i,j=1}^{m}\sum_{s,t=1}^{q}(tr_{q}\otimes\tau)(k_{st}f_{ji}\otimes y_{st}x_{ij}^{*})\\
&=\sum_{i,j=1}^{m}\sum_{k=1}^{min\{p_{i},p_{j}\}}\sum_{s,t=1}^{q}tr_{q}(k_{st}k_{p_{1}+p_{2}+\ldots+p_{j-1}+k,~p_{1}+p_{2}+\ldots+p_{i-1}+k})\\
&\circ\tau (y_{st}x_{ij}^{*})\\
&=\sum_{i,j=1}^{m}\sum_{k=1}^{min\{p_{i},p_{j}\}}\frac{\tau(y_{p_{1}+p_{2}+\ldots+p_{i-1}+k,~p_{1}+p_{2}+
 \ldots+p_{j-1}+k} x^{*}_{ij})}{q}.
\end{align*}

By Lemma \ref{3}, there exists a unique normal u.c.p. map $S$ from $M_{q}(\mathbb{C})\otimes N$ to $M_{m}(\mathbb{C})\otimes N$ such that
 for $x\in M_{m}(\mathbb{C})\otimes N,~y\in M_{q}(\mathbb{C})\otimes N$, $(tr_{m}\otimes\tau)(h^{\frac{1}{2}}S(y) h^{\frac{1}{2}}x^{*})
 =(tr_{q}\otimes\tau)(yT(x)^{*}),$
so it follows that $S=S'$ and $\phi\circ S=tr_{q}\otimes\tau$.

Since $T(e_{ij}\otimes x_{ij})=f_{ij}\otimes x_{ij}$, by the definition of $S'=S$
we have
$$S\circ T(e_{ij}\otimes x_{ij})=\frac{min\{p_{i},p_{j}\}}{\sqrt{p_{i}p_{j}}}e_{ij}\otimes x_{ij}.$$
This ends the proof of Claim B.

Now we check that $\|S\circ T(x)-x\|_{\phi}^{\sharp}\leq\|h^{\frac{1}{2}}x-xh^{\frac{1}{2}}\|_{2},~x\in M_{m}(\mathbb{C})\otimes N$.

For any $x=\sum_{i,j=1}^{m}x_{ij}\otimes e_{ij}\in M_{m}(\mathbb{C})\otimes N$,
\begin{align*}
(\|x\|_{\phi}^{\sharp})^{2}&=\phi(\frac{xx^{*}+x^{*}x}{2})\\
&=(tr_{m}\otimes\tau)(\frac{h(xx^{*}+x^{*}x)}{2})\\
&=\frac{1}{2m}\sum_{i,j=1}^{m}(\lambda_{i}+\lambda_{j})\|x_{ij}\|_{2}^{2}\\
&=\frac{1}{2q}\sum_{i,j=1}^{m}(p_{i}+p_{j})\|x_{ij}\|_{2}^{2}.\\
\end{align*}

Hence $(\|S\circ T(x)-x\|_{\phi}^{\sharp})^{2}=\frac{1}{2q}\sum_{i,j=1}^{m}(p_{i}+p_{j})(1-\frac{min\{p_{i},
p_{j}\}}{\sqrt{p_{i}p_{j}}})^{2}\|x_{ij}\|_{2}^{2}.$

If $p_{i}\leq p_{j},$
\begin{align*}
(1-\frac{min\{p_{i},p_{j}\}}{\sqrt{p_{i}p_{j}}})^{2}&=(1-(\frac{p_{i}}{p_{j}})^{\frac{1}{2}})^{2}\\
&=\frac{1}{p_{j}}(p^{\frac{1}{2}}_{i}-p^{\frac{1}{2}}_{j})^{2}\\
&\leq\frac{2}{p_{i}+p_{j}}(p^{\frac{1}{2}}_{i}-p^{\frac{1}{2}}_{j})^{2}.\\
\end{align*}

By symmetry, the formula also holds for $p_{j}\leq p_{i}$. Hence
$$(\|S\circ T(x)-x\|_{\phi}^{\sharp})^{2}\leq\frac{1}{q}\sum_{i,j=1}^{m}(p^{\frac{1}{2}}_{i}-p^{\frac{1}{2}}_{j})^{2}\|x_{ij}\|_{2}^{2}.$$

On the other hand, the $(i,j)$-th element of the matrix $h^{\frac{1}{2}}x-xh^{\frac{1}{2}}$ is $(\lambda_{i}^{\frac{1}{2}}-\lambda_{j}
^{\frac{1}{2}})x_{ij}$. Thus
\begin{align*}
\|h^{\frac{1}{2}}x-xh^{\frac{1}{2}}\|_{2}^{2}&=\frac{1}{m}\sum_{i,j=1}^{m}(\lambda_{i}^{\frac{1}{2}}-\lambda_{j}^{\frac{1}{2}})^{2}\|x_{ij}\|_{2}^{2}\\
&=\frac{1}{q}\sum_{i,j=1}^{m}(p_{i}^{\frac{1}{2}}-p_{j}^{\frac{1}{2}})^{2}\|x_{ij}\|_{2}^{2}.\\
\end{align*}

Then we finish the proof.
\end{proof}

With the help of the above four lemmas, we now proceed to prove Theorem \ref{diffuse}. Actually, the proof of Theorem \ref{diffuse} is adapted from \cite[Lemma 3.4, Proposition 3.5]{haagerup}. For the reader's convenience, we include the proof below.
\begin{proof}[Proof of Theorem \ref{diffuse}]
It is sufficient to consider unitary operators $u_{1},\ldots,u_{n}\in M$.

Claim A: there exists a $q\in \mathbb{N}$, a normal u.c.p. map $T$ from $M_{q}(\mathbb{C})\otimes N$ to $M$, and $n$ operators
 $y_{1},\ldots,y_{n}\in M_{q}(\mathbb{C})\otimes N$, such that $\|y_{k}\|\leq1$, $\tau\circ T=tr_{q}\otimes\tau$
and $\|T(y_{k})-u_{k}\|_{2}<\varepsilon,~k=1,\ldots,n.$

\emph{Proof of Claim A}.
Let $\varepsilon>0$. By Lemma \ref{2}, there exists an $m\in \mathbb{N}$, and normal u.c.p. maps
$S_{1}:M\rightarrow M_{m}(\mathbb{C})\otimes N$ and $T_{1}:M_{m}(\mathbb{C})\otimes N\rightarrow M$ such that
$\|T_{1}\circ S_{1}(u_{k})-u_{k}\|_{2}<\varepsilon,~k=1,\ldots,n$, and $\tau\circ T_{1}(x)=(tr_{m}\otimes\tau)(hx)$, where $h$ is an invertible rational positive ``good'' simple operator, which is of scalar form.
Put $x_{k}=S_{1}(u_{k}), ~k=1,\ldots,n$. Note that $\|x_{k}\|\leq1$ and $$\|T_{1}(x_{k})-u_{k}\|_{2}<\varepsilon,~k=1,\ldots,n.$$

Put $\phi(x)=(tr_{m}\otimes\tau)(hx),~x\in M_{m}(\mathbb{C})\otimes N.$
By Lemma \ref{4}, there exists a $q\in \mathbb{N}$, normal u.c.p. maps $T_{2}:M_{m}(\mathbb{C})\otimes N\rightarrow M_{q}(\mathbb{C})\otimes N$ and
$S_{2}:M_{q}(\mathbb{C})\otimes N\rightarrow M_{m}(\mathbb{C})\otimes N$ such that
$\phi\circ S_{2}=tr_{q}\otimes\tau,~(tr_{q}\otimes\tau)\circ T_{2}=\phi$,
and $\|S_{2}\circ T_{2}(x)-x\|_{\phi}^{\sharp}\leq\|h^{\frac{1}{2}}x-xh^{\frac{1}{2}}\|_{2},~x\in M_{m}(\mathbb{C})\otimes N.$

For $k=1,\ldots,n$,
\begin{align*}
\|h^{\frac{1}{2}}x_{k}-x_{k}h^{\frac{1}{2}}\|_{2}^{2}&=(tr_{m}\otimes\tau)(hx_{k}x_{k}^{*}+hx_{k}^{*}x_{k}-2h^{\frac{1}{2}}x_{k}h^{\frac{1}{2}}x_{k}^{*})\\
&=\phi(x_{k}x_{k}^{*})+\phi(x_{k}^{*}x_{k})-2(tr_{m}\otimes\tau)(h^{\frac{1}{2}}x_{k}h^{\frac{1}{2}}x_{k}^{*})\\
&\leq2-2(tr_{m}\otimes\tau)(h^{\frac{1}{2}}x_{k}h^{\frac{1}{2}}x_{k}^{*}).
\end{align*}

By Lemma \ref{3} (2),
\begin{align*}
(tr_{m}\otimes\tau)(h^{\frac{1}{2}}x_{k}h^{\frac{1}{2}}x_{k}^{*})&\geq \|T_{1}(x_{k})\|_{2}^{2}\\
&\geq(\|u_{k}\|_{2}-\|u_{k}-T_{1}(x_{k})\|_{2})^{2}\\
&>(1-\varepsilon)^{2}\\
&> 1-2\varepsilon.\\
\end{align*}
Then we have $\|S_{2}\circ T_{2}(x_{k})-x_{k}\|_{\phi}^{\sharp}<2\varepsilon^{\frac{1}{2}}.$

Put $y_{k}=T_{2}(x_{k}), ~k=1,\ldots,n$ and $T=T_{1}\circ S_{2}$. Then $T$ is a normal u.c.p. map such that
$\tau\circ T=(\tau\circ T_{1})\circ S_{2}=\phi\circ S_{2}=tr_{q}\otimes\tau.$

By the Schwartz inequality for c.p. maps, we have for $x\in M_{m}(\mathbb{C})\otimes N$,
\begin{align*}
\|T_{1}(x)\|_{2}^{2}&\leq \frac{1}{2}\tau(T_{1}(x^{*}x)+T_{1}(xx^{*}))\\
&=(\|x\|_{\phi}^{\sharp})^{2}.
\end{align*}

Note that
\begin{align*}
\|T(y_{k})-T_{1}(x_{k})\|_{2}&=\|T_{1}(S_{2}(y_{k})-x_{k})\|_{2}\\
&\leq \|S_{2}(y_{k})-x_{k}\|_{\phi}^{\sharp}\\
&<2\varepsilon^{\frac{1}{2}}.
\end{align*}
Then we have
$\|T(y_{k})-u_{k}\|_{2}<3\varepsilon^{\frac{1}{2}},~k=1,\ldots,n.$
This ends the proof of Claim A.

By Lemma \ref{3} (1), there is a unique normal u.c.p. map $S$ from $M$ to $M_{q}(\mathbb{C})\otimes N$
such that $(tr_{q}\otimes\tau)(S(y)x^{*})=\tau(yT(x)^{*})$, for $y\in M$, $x\in M_{q}(\mathbb{C})\otimes N,$
and $(tr_{q}\otimes\tau)\circ S=\tau$.

Note that
\begin{align*}
\|T(x)\|_{2}^{2}&\leq \tau(T(x^{*}x))\\
&=(tr_{q}\otimes\tau)(x^{*}x)\\
&=\|x\|_{2}.\\
\end{align*}
Similarly we get $\|S(y)\|_{2}\leq \|y\|_{2}$, $y\in M$.

For $k=1,\ldots,n$,
\begin{align*}
|(tr_{q}\otimes\tau)(S(u_{k})y_{k}^{*})|&=|\tau(u_{k}T(y_{k})^{*})|\\
&=|\tau(1)-\tau(u_{k}(u_{k}-T(y_{k}))^{*})|\\
&\geq 1-\|u_{k}\|_{2}\|u_{k}-T(y_{k})\|_{2}\\
&>1-3\varepsilon^{\frac{1}{2}},\\
\end{align*}

\begin{align*}
\textmd{Im~}\tau(u_{k}T(y_{k})^{*})&=\frac{1}{2}|\tau(u_{k}T(y_{k})^{*})-\tau(u_{k}^{*}T(y_{k}))|\\
&=\frac{1}{2}|\tau(u_{k}(T(y_{k})-u_{k})^{*})-\tau(u_{k}^{*}(T(y_{k})-u_{k}))|\\
&\leq\|T(y_{k})-u_{k}\|_{2}\\
&<3\varepsilon^{\frac{1}{2}}.\\
\end{align*}
Then we conclude that $\textmd{Re~}\tau(u_{k}T(y_{k})^{*})>\sqrt{(1-3\varepsilon^{\frac{1}{2}})^2-(3\varepsilon^{\frac{1}{2}})^2}>1-6\varepsilon^{\frac{1}{2}}$.

Thus, we obtain that
\begin{align*}
\|S(u_{k})-y_{k}\|_{2}^{2}&=\|S(u_{k})\|_{2}^{2}+\|y_{k}\|_{2}^{2}-2\textmd{Re}(tr_{q}\otimes\tau)(S(u_{k})y_{k}^{*})\\
&<2-2(1-6\varepsilon^{\frac{1}{2}})\\
&=12\varepsilon^{\frac{1}{2}}.\\
\end{align*}

Hence,
\begin{align*}
\|T\circ S(u_{k})-u_{k}\|_{2}&=\|T(S(u_{k})-y_{k})\|_{2}+\|T(y_{k})-u_{k}\|_{2}\\
&<4\varepsilon^{1/4}+3\varepsilon^{1/2}.
\end{align*}
\end{proof}

\section{Permanence properties for amenable inclusions}
In this section, we apply our main theorems to study permanence properties for amenable inclusions.

\subsection*{Haagerup property}
In \cite{jo}, it was shown that if the basic construction $\langle M,e_{N}\rangle$ is a finite von Neumann algebra and $N$ has the Haagerup
property, then $M$ also has the Haagerup property. Anantharaman-Delaroche \cite{anan6} showed that if L$H\subset$L$G$ is an amenable inclusion of group von Neumann algebras and L$H$ has the Haagerup property, then L$G$ also has the Haagerup property. In \cite{popa2}, Popa asked if the inclusion of finite von Neumann algebras $N\subset M$ is amenable, and $N$ has the Haagerup property, does $M$ also have the Haagerup property?
Bannon and Fang settled the question in the affirmative in \cite{fang}. Their proof is based on an equivalent characterization of the Haagerup property using correspondences.

Since the definition of the Haagerup property involves normal c.p. maps, it is natural to expect a proof using normal c.p. maps rather than correspondences. As an application of our main results, we can give such a proof of certain cases of Bannon-Fang's result.

\begin{corollary}\label{corollary on permanence results}
Let $M$ be a finite von Neumann algebra (resp. a type II$_{1}$ factor) with a faithful normal tracial state $\tau$, and let $N\subset M$ be a type II$_{1}$ factor (resp. $N$ have an atomic part). If the inclusion $N\subset M$ is amenable and $N$ has the Haagerup property, then $M$ also has the Haagerup property.
\end{corollary}
\begin{proof}
 Let $\{x_{1},\ldots,x_{n}\}$ be a finite set in $M$ and let $\varepsilon>0$. By Theorem \ref{diffuse} ( resp. Theorem \ref{atomic}), there exists an $m\in \mathbb{N}$, and normal u.c.p. maps
$S:M\rightarrow M_{m}(\mathbb{C})\otimes N,\quad T:M_{m}(\mathbb{C})\otimes N\rightarrow M,$
such that $(tr_{m}\otimes\tau)\circ S=\tau,\quad \tau\circ T=tr_{m}\otimes\tau$
and $\|T\circ S(x_{i})-x_{i}\|_{2}<\varepsilon$, $i=1,\ldots,n$. Since $N$ has the Haagerup property, we can find a normal c.p. map $L: M_{m}(\mathbb{C})\otimes N\rightarrow  M_{m}(\mathbb{C})\otimes N$, such that $(tr_{m}\otimes\tau)\circ L\leq tr_{m}\otimes\tau$, $\|L(S(x_{i}))-S(x_{i})\|_{2}<\varepsilon$, $i=1,\ldots,n$, and
$L$ induces a compact bounded operator on $L^{2}(M)$. It is easy to check that $T\circ L\circ S$ satisfies the subtracial condition $\tau\circ T\circ L\circ S\leq \tau$, and it induces a compact bounded operator on $L^{2}(M)$. Moreover, we have
\begin{align*}
 \|T\circ L\circ S(x_{i})-x_{i}\|_{2}&=\|T\circ L\circ S(x_{i})-T\circ S(x_{i})+T\circ S(x_{i})-x_{i}\|_{2}\\
&\leq \|T\|\| L\circ S(x_{i})-S(x_{i})\|_{2}+\|T\circ S(x_{i})-x_{i}\|_{2}\\
&<2\varepsilon.
\end{align*}

Let $\Lambda=\{(E,\varepsilon):~E~\text{is a finite subset in M and~}\varepsilon>0\}$. For $(E,\varepsilon)$, $(F,\epsilon)\in \Lambda$, define $(E,\varepsilon)\prec(F,\epsilon)$ if $E\subseteq F$ and $\varepsilon\geq \epsilon$. Then $\Lambda$ is a directed set. Thus $(T\circ L\circ S_{(\{x_{1},\ldots,x_{n}\},\varepsilon)})_{(\{x_{1},\ldots,x_{n}\},\varepsilon)\in \Lambda}$ is the net which proves the corollary.
\end{proof}

\subsection*{Weak Exactness}
The theory of exact $C^{*}$-algebras was introduced and studied intensively by Kirchberg. It has been playing a significant role in the development of $C^{*}$-algebras, e.g. in the classification of $C^{*}$-algebras (see \cite{kirchberg,rodan}) and in the theory of noncommutative topological entropy (see \cite{Brown,stomer,voiculescu}). Hence it is natural to explore an analogue of this notion for von Neumann algebras. The concept of weakly exact von Neumann algebras was also introduced by Kirchberg \cite{kirchberg}. He proved that a von Neumann algebra $M$ is weakly exact if it contains a dense weakly exact $C^{*}$-algebra. Ozawa in \cite{ozawa} gave a local characterization of weak exactness and proved that a discrete group is exact if and only if its group von Neumann algebra is weakly exact. Weak exactness also passes to a von Neumann subalgebra which is the range of a normal conditional expectation. Hence, every von Neumann subalgebra of a weakly exact finite von Neumann algebra is again weakly exact. It is left open
whether the ultrapower $R^{\omega}$ of the hyperfinite type II$_{1}$ factor $R$ is weakly exact or not. For more details and results on weak exactness, we refer the reader to \cite{ozawabook,ozawa}.

As the second application of our main results Theorem \ref{diffuse} and Theorem \ref{atomic}, we prove a permanence result for weak exactness.
\begin{corollary}\label{weakly}
Let $M$ be a finite von Neumann algebra (resp. a type II$_{1}$ factor) with a faithful normal tracial state $\tau$, and let $N\subset M$ be a type II$_{1}$ factor (resp. $N$ have an atomic part). If the inclusion $N\subset M$ is amenable and $N$ is weakly exact, then $M$ is also weakly exact.
\end{corollary}
\begin{proof}
Let $E$ be a finite dimensional operator system in $M$. Since the inclusion $N\subset M$ is amenable, by Theorem  \ref{diffuse} (resp. Theorem \ref{atomic}), there exist two nets of trace preserving normal u.c.p. maps $ S_i:M\rightarrow M_{n_{i}}(\mathbb{C})
\otimes N$ and $ T_i:M_{n_{i}}(\mathbb{C})\otimes N\rightarrow M$, such that for all $x\in M$,
$ T_i\circ S_i(x)\rightarrow x$ in the $\|\cdot\|_{2}$-norm topology. By \cite[Corollary 14.1.5]{ozawabook}, $M_{n_{i}}(\mathbb{C})\otimes N$ is weakly exact. Note that $ S_i(E)\subset \widetilde{E}$ for some finite-dimensional operator system $\widetilde{E}$ in $M_{n_{i}}(\mathbb{C})\otimes N$. By \cite[p.2]{ozawa} and Remark \ref{different topologies coincide}, there exist two nets of u.c.p. maps $ S^{'}_j:\widetilde{E}\rightarrow M_{n_i}(\mathbb{C})$ and $T'_{j}: S^{'}_j(\widetilde{E})\rightarrow M_{n_{i}}(\mathbb{C})\otimes N$ such that the net $( T^{'}_j\circ S^{'}_j)$ converges to $id_{\widetilde{E}}$ in the point-wise $\|\cdot\|_{2}$-norm topology.  For $x\in E$, we have
\begin{align*}
 \| T_i\circ  T^{'}_j\circ  S^{'}_j\circ  S_i(x)-x\|_{2}&\leq \| T_i( T^{'}_j\circ S^{'}_j\circ  S_i(x)- S_i(x))\|_{2}+
 \| T_i\circ  S_i(x)-x\|_{2}\\
&\leq\| T^{'}_j\circ S^{'}_j\circ  S_i(x)- S_i(x)\|_{2}+
 \| T_i\circ  S_i(x)-x\|_{2}\\
&\rightarrow 0.
\end{align*}
The second inequality follows from the fact that $ T_i$ is a trace preserving u.c.p. map.
Thus $(S_i\circ  S^{'}_j)$ and $(T^{'}_j\circ T_i)$ are two nets of u.c.p. maps witnessing the weak exactness of $M$.
\end{proof}

\subsection*{Weak Haagerup property}
In \cite{Knudby}, the author introduced the weak Haagerup property both for locally compact groups and finite von Neumann algebras. He proved that a discrete group has the weak Haagerup property if and only if its group von Neumann algebra does and several hereditary results for the weak Haagerup property. We should mention that the weak Haagerup property of a von Neumann algebra does not depend on the choice of faithful normal traces by \cite[Proposition 8.4]{Knudby}, hence we omit the mention of the trace below.

Note that the weak Haagerup property requires normal completely bounded maps. Our main results give a description of relative amenability using normal unital completely positive maps, which are naturally completely bounded. Thus, as the third application of our main results, we add one more permanence property.
\begin{corollary}\label{weak}
Let $M$ be a finite von Neumann algebra (resp. a type II$_{1}$ factor) with a faithful normal tracial state $\tau$, and let $N\subset M$ be a type II$_{1}$ factor (resp. $N$ have an atomic part). If the inclusion $N\subset M$ is amenable and $N$ has the weak Haagerup property, then $M$ also has the weak Haagerup property.
\end{corollary}
\begin{proof}
Let $\{x_{1},\ldots,x_{n}\}$ be a finite set in the unit ball of $M$ and let $\varepsilon>0$. By Theorem \ref{diffuse} (resp. Theorem \ref{atomic}), there exists an $m\in \mathbb{N}$, and two normal u.c.p. maps
$S:M\rightarrow M_{m}(\mathbb{C})\otimes N,$ $T:M_{m}(\mathbb{C})\otimes N\rightarrow M,$
such that $(tr_{m}\otimes\tau)\circ S=\tau$, $\tau\circ T=tr_{m}\otimes\tau$ and $\|T\circ S(x_{k})-x_{k}\|_{2}<\varepsilon,~k=1,\ldots,n$.
By \cite[Lemma 2.5]{anan7}, there exist two normal u.c.p. maps $S': M_m(\mathbb{C})\otimes N\to M$ and $T': M\to M_{m}(\mathbb{C})\otimes N$ such that $\langle S(x),a \rangle_{tr_m\otimes \tau}=\langle x, S'(a) \rangle_{\tau}$ and $\langle T(a), y\rangle_{\tau}=\langle a, T'(y) \rangle_{tr_m\otimes \tau}$ for all $x, y\in M$ and $a\in M_m(\mathbb{C})\otimes N$. Since $N$ has the weak Haagerup property, there exists a constant $C>0$ and a normal completely bounded map $L$ on $M_m(\mathbb{C})\otimes N$ with $\|L\|_{c.b.}\leq C$ such that $\langle L(a),b\rangle_{tr_m\otimes \tau}=\langle a,L(b)\rangle_{tr_m\otimes \tau}$ for $a,~b\in M_m(\mathbb{C})\otimes N$, $L$ induces a compact bounded map on $L^2(M_m(\mathbb{C})\otimes N)$, and for $i,j=1,\ldots,n$, $|\langle L\circ S(x_{i})- S(x_{i}),T'(x_{j})\rangle_{tr_{m}\otimes\tau}|<\varepsilon$, $|\langle L\circ T'(x_{i})- T'(x_{i}),S(x_{j})\rangle_{tr_{m}\otimes\tau}|<\varepsilon$ following from \cite[Remark 7.5]{Knudby}.

Define $\widetilde{T}=\frac{1}{2}(T\circ L\circ S+S'\circ L\circ T')$.
It is clear that $\widetilde{T}$ is a normal completely bounded map with $\|\widetilde{T}\|_{c.b.}\leq C$, since $T,T',S,S'$ are normal u.c.p. maps and $L$ is a normal completely bounded map with $\|L\|_{c.b.}\leq C$.

We check that $\langle \widetilde{T}(x), y\rangle_{\tau}=\langle x, \widetilde{T}(y)\rangle_{\tau}$ for $x,y\in M$.
Note that
\begin{align*}
 \langle T\circ L\circ S(x),y \rangle_{\tau}&=\langle L\circ S(x),T'(y) \rangle_{tr_{m}\otimes\tau}\\
&=\langle S(x),L\circ T'(y) \rangle_{tr_{m}\otimes\tau}\\
&=\langle x,S'\circ L\circ T'(y) \rangle_{\tau}.
\end{align*}
Clearly, this implies $\langle \widetilde{T} (x), y\rangle_{\tau}=\langle x, \widetilde{T}(y)\rangle_{\tau}$ for $x,y\in M$.
It is easy to see that $\widetilde{T}$ induces a compact operator on $L^{2}(M)$, since $L$ induces a compact operator.

We check that $|\langle\widetilde{T}x_{i}- x_{i},x_{j}\rangle_{\tau}|<2\varepsilon$ for $i,j=1,\ldots,n$.
Since $\|T\circ S(x_{i})-x_{i}\|_{2}<\varepsilon$ and $x_{i}$ is in the unit ball of $M$, it follows that $|\langle T\circ S (x_{i})- x_{i},x_{j}\rangle_{\tau}|<\varepsilon$, for $i,j=1,\ldots,n$.

Thus we have
\begin{align*}|\langle T\circ L\circ S(x_{i})-x_{i},x_{j} \rangle_{\tau}|&=|\langle L\circ S(x_{i}),T'(x_{j}) \rangle_{tr_{m}\otimes\tau}-\langle x_{i},x_{j} \rangle_{\tau}|\\
&\leq|\langle L\circ S(x_{i})- S(x_{i}),T'(x_{j})\rangle_{tr_{m}\otimes\tau}+\langle S(x_{i}),T'(x_{j})\rangle_{tr_{m}\otimes\tau}\\
&\quad-\langle x_{i},x_{j} \rangle_{\tau}|\\
&<2\varepsilon.
\end{align*}
Similarly,
\begin{align*}
 |\langle S'\circ L\circ T'(x_{i})-x_{i},x_{j} \rangle_{\tau}|&=|\langle L\circ T'(x_{i}),S(x_{j}) \rangle_{tr_{m}\otimes\tau}-\langle x_{i},x_{j} \rangle_{\tau}|\\
&\leq|\langle L\circ T'(x_{i})- T'(x_{i}),S(x_{j})\rangle_{tr_{m}\otimes\tau}\\
&\quad+\langle T'(x_{i}),S(x_{j})\rangle_{tr_{m}\otimes\tau}-\langle x_{i},x_{j} \rangle_{\tau}|\\
&<2\varepsilon.
\end{align*}
Let $\Lambda=\{(E,\varepsilon):~E~\text{is a finite subset in the unit ball of M and~}\varepsilon>0\}$. For $(E,\varepsilon)$, $(F,\epsilon)\in \Lambda$, define $(E,\varepsilon)\prec(F,\epsilon)$ if $E\subseteq F$ and $\varepsilon\geq \epsilon$. Then $\Lambda$ is a directed set. Thus $(\widetilde{T}_{(\{x_{1},\ldots,x_{n}\},\varepsilon)})_{(\{x_{1},\ldots,x_{n}\},\varepsilon)\in \Lambda}$ is the net which proves the corollary.
\end{proof}

\subsection*{Concluding remark}
Recall that a type II$_{1}$ factor $M$ with a trace $\tau$ is said to have property $\Gamma$ if, given any $\varepsilon>0$ and $x_{1},\ldots,x_{n}\in M$, there exists a trace zero unitary $u\in M$ such that $\|ux_{i}-x_{i}u\|_{2}<\varepsilon$, $1\leq i\leq n$.
In \cite[Problem 3.3.2]{popa}, Popa asked, if $N\subset M$ are type II$_{1}$ factors with trace $\tau$, the inclusion $N\subset M$ is amenable, and $N$ has property $\Gamma$, does this imply that $M$ has property $\Gamma$? In \cite{anan3}, B\'edos proved that if $G$ is a discrete amenable group with a free action $\alpha$ on a von Neumann algebra $N$ and $N$ has property $\Gamma$, then $M:=N\rtimes_{\alpha}G$ has property $\Gamma$.
We tried to use our Theorem \ref{diffuse} to attack this problem, but did not succeed. The reason is as follows. Following the above ideas, assume $x_{1},\ldots,x_{n}$ are finite elements in the unit ball of $M$. By Theorem \ref{diffuse}, for any $\varepsilon>0$, there exists an $m\in \mathbb{N}$, and two normal u.c.p. maps
$S:M\rightarrow M_{m}(\mathbb{C})\otimes N,$ $T:M_{m}(\mathbb{C})\otimes N\rightarrow M,$
such that $(tr_{m}\otimes\tau)\circ S=\tau$, $\tau\circ T=tr_{m}\otimes\tau$ and $\|T\circ S(x_{k})-x_{k}\|_{2}<\varepsilon,~k=1,\ldots,n$. Since $N$ has property $\Gamma$, we can find a unitary operator $\tilde{u}\in M_{m}(\mathbb{C})\otimes N$  with $(tr_{m}\otimes \tau)(\tilde{u})=0$ such that $\|S(x_{i})\tilde{u}-\tilde{u}S(x_{i})\|_{2}<\varepsilon$. It follows that $\|T(S(x_{i})\tilde{u}-\tilde{u}S(x_{i}))\|_{2}<\varepsilon$ and $\tau\circ T(\tilde{u})=(tr_{m}\otimes \tau)(\tilde{u})=0$, since $T$ is a trace preserving normal u.c.p map. Then, we run into two problems. One is that this normal u.c.p. map $T$ is not a homomorphism on the algebra $M_{m}(\mathbb{C})\otimes N$. If so, then we would have $\|x_{i}T(\tilde{u})-T(\tilde{u})x_{i}\|_{2}<2\varepsilon$, $1\leq i\leq n$ and $\tau\circ T(\tilde{u})=0$, but we don't know this $T(\tilde{u})$ is a unitary operator or not, or it can be approximated by trace zero unitaries in $M$.

\proof[Acknowledgements]
The first author would like to thank Yongle Jiang for providing several helpful suggestions and comments.
The second author was supported by the Project sponsored by the NSFC grant 11431011 and startup funding from Hebei Normal University.
The authors would like to thank the referee for several useful comments.


\bibliographystyle{amsplain}

\end{document}